\documentclass[hidelinks,onefignum,onetabnum]{siamart251104}

\usepackage{lipsum}
\usepackage{amsfonts}
\usepackage{graphicx}
\usepackage{epstopdf}
\usepackage{algorithmic}
\ifpdf
  \DeclareGraphicsExtensions{.eps,.pdf,.png,.jpg}
\else
  \DeclareGraphicsExtensions{.eps}
\fi

\newsiamremark{remark}{Remark}
\newsiamremark{hypothesis}{Hypothesis}
\crefname{hypothesis}{Hypothesis}{Hypotheses}
\newsiamthm{claim}{Claim}
\newsiamremark{fact}{Fact}
\crefname{fact}{Fact}{Facts}

\headers{The discrete gradient flow for MBECs}{Z. Feng, L. Liu, Q. Tang}

\title{Discrete normalized gradient flow for two-component Bose-Einstein condensates: Energy dissipation, global convergence and sharp local convergence behavior  \thanks{Submitted to the editors DATE.
\funding{This work was funded by the National Key R\&D Program of China (No.2024YFA1012803) and the Natural Science Foundation of Sichuan Province (Grant No.2024NSFSC0438).}}}

\author{
Zixu Feng\thanks{
School of Mathematical Science, Chengdu University of Technology, Chengdu 610059, P. R. China
(\email{zixu\_feng123@163.com})
}
\and
Lunxu Liu\thanks{
School of Mathematics, Sichuan University, Chengdu, P. R. China
(\email{liulunxu6@stu.scu.edu.cn},
\email{qinglin\_tang@scu.edu.cn})
}
\and
Qinglin Tang\footnotemark[3]
}

\usepackage{amsopn}

\usepackage{xcolor} 
\usepackage{array}
\usepackage{booktabs}
\usepackage{bm}
\newtheorem{example}{Example}[section]

\ifpdf
\hypersetup{
  pdftitle={The discrete gradient flow for MBECs},
  pdfauthor={Z. Feng, L. Liu, Q. Tang}
}
\fi

\begin{document}

\maketitle

\begin{abstract}
 The gradient flow with semi-implicit discretization (GFSI) is the most widely used algorithm for computing the ground state of Gross-Pitaevskii energy functional. We apply GFSI to the two-component scenario with Josephson junction and rotating term, which is one of the most important and topical models in multi-component Bose-Einstein condensates (MBECs), and rigorously establish the following fundamental results for the first time. By introducing a Lagrange multiplier to reformulate GFSI into an equivalent form, we prove its energy dissipation property and global convergence to stationary states. More significantly, we uncover an intrinsic connection between this classical numerical PDE discretization rooted in imaginary-time evolution and Riemannian optimization, a state-of-the-art mathematical framework for manifold-constrained optimization. This connection enables us to fully characterize the local convergence behavior of GFSI within the Riemannian optimization framework. Together with the aforementioned global convergence result, this yields a complete global--local convergence theory for GFSI. Finally, numerical experiments comprehensively validate the theoretically predicted energy dissipation and convergence properties.
 

 %
\end{abstract}


\section{Introduction}
Bose-Einstein condensates (BECs), a macroscopic quantum phenomenon in which bosons occupy the same quantum ground state at ultra-low temperatures\cite{ref_low_temp}, have become an ideal platform for the study of quantum many-body interactions and nonlinear dynamics since its experimental realization in 1995\cite{ref_Abo,ref_S.K.,ref_And,ref_A.1,ref_A.2}. There is particular interest in creating long-lived multi-component BEC systems, where the condensate wave functions are affected by inter-component interactions\cite{ref_2-BECS}. 

Mathematically, the behavior of two-component BECs can be described by a complex-valued wave function $\Psi:=(\psi_1,\psi_2)^{\top}$, defined on $\mathbb{R}^d$, $d=2,3$. The coupled Gross-Pitaevskii energy functional for two-component BECs with Josephson junction and rotating term is given by 
\begin{equation}\label{original total energy}
\begin{aligned}
E(\Psi) & =\sum_{i=1,2}  \int_{\mathbb{R}^d}\frac{1}{2}|\nabla\psi_i|^2+V_i(\mathbf{x})|\psi_i|^2+\frac{1}{2}(k_{i1}|\psi_1|^2+k_{i2}|\psi_2|^2) |\psi_i|^2 \\
& \qquad - \omega_i\overline{\psi_i}L_z\psi_i \mathrm{d}\mathbf{x}  + 2\beta \int_{\mathbb{R}^d}\mathrm{Re}(\psi_1\overline{\psi_2})\mathrm{d}\mathbf{x}.
\end{aligned}
\end{equation}

Here, $\mathbf{x}\in \mathbb{R}^d$ is the spatial coordinate with $\mathbf{x}=(x,y)^{\top}$ in 2D and $\mathbf{x}=(x,y,z)^{\top}$ in 3D. $V_i$, $i=1,2$ denote the real-valued external potential,  $\beta$ describes the strength of internal atomic Josephson junction and density functions $|\psi_i|^2$, $i=1,2$ are coupled by the linear combinations:
\begin{equation}\label{density}
    \rho_1(\Psi)=k_{11}|\psi_1|^2+k_{12}|\psi_2|^2, \quad \rho_2(\Psi)=k_{21}|\psi_1|^2+k_{22}|\psi_2|^2,
\end{equation}
where $k_{ij}$ represents the strength of particle interactions between the $i$-th and $j$-th condensate components ($k_{ij}>0$ for the repulsive case, $k_{ij}<0$ for the attractive case). The interaction matrix $(k_{ij})_{2 \times 2}$ is symmetric, i.e., $k_{12}=k_{21}$. $L_z$ is the angular momentum operator defined as 
\[ L_z:=-i(x\partial y - y \partial x),\]
and $\omega:=(\omega_1,\omega_2)^{\top}$ denoting the non-negative rotation frequency.

The components of the wave function $\Psi$ are square integrable with respect to spatial coordinates, with the squares of their $L^2$-norms representing the total mass of each component. This total mass is conserved, satisfying the constraint
\begin{equation}
M(\Psi):=\|\psi_1\|^2_{L^2}+\|\psi_2\|^2_{L^2} \equiv 1.
\end{equation} 
If there is no internal Josephson junction, i.e. $\beta = 0$, then the mass of each component is also conserved\cite{ref_bao_cai,ref_Kinet}, i.e. 
$$ \|\psi_1\|^2_{L^2} = M_1, \quad \|\psi_2\|^2_{L^2} = 1-M_1, $$
with $0 \leq M_1 \leq 1$ a given constant.

A central problem in the study of BECs is the computation of the ground state, defined as a wave function $\Psi_g(\mathbf{x})$ that minimizes the functional $E(\Psi)$ under the mass constraint:
\begin{equation}\label{original mini problem}
\Psi_g (\mathbf{x}) := \arg\min_{\Psi \in \mathcal{M}} E(\Psi) ,
\end{equation}
where $\mathcal{M} := \left\{ \Psi \in [L^2(\mathbb{R}^d)]^2 \mid \|\psi_1\|^2_{L^2}+\|\psi_2\|^2_{L^2}=1, E(\Psi) < \infty \right\}$.
Also we have the Euler-Lagrange equation associated with the above minimization problem reads as 
\begin{equation}\label{lageq}
    \begin{aligned}
     & \mu\psi_1 = \mathcal{H}_{\psi_1}\psi_1+ \beta \psi_2, \quad \mathcal{H}_{\psi_1}:= -\frac{1}{2}\nabla^2 +V_1(\mathbf{x})+\rho_1(\Psi)-\omega_1 L_z, \\
     & \mu\psi_2 = \mathcal{H}_{\psi_2}\psi_2 + \beta \psi_1, \quad \mathcal{H}_{\psi_2}:= -\frac{1}{2}\nabla^2 +V_2(\mathbf{x})+\rho_2(\Psi)-\omega_2 L_z, \quad \mathbf{x} \in \mathbb{R}^{d},
    \end{aligned}
\end{equation}
under the mass constraint $\Psi \in \mathcal{M}$. This is a nonlinear eigenvalue problem for $(\mu,\Psi)$, the corresponding eigenvalue $\mu$ can be computed through the eigenfunction $\Psi$ by
\[ \begin{aligned}
 \mu(\Psi) & =\sum_{i=1,2}  \int_{\mathbb{R}^d}\frac{1}{2}|\nabla\psi_i|^2+V_i(\mathbf{x})|\psi_i|^2+\rho_i(\Psi)|\psi_i|^2 - \omega_i\overline{\psi_i}L_z\psi_i \mathrm{d}\mathbf{x}   \\
 & \qquad + 2\beta \int_{\mathbb{R}^d}\mathrm{Re}(\psi_1\overline{\psi_2})\mathrm{d}\mathbf{x} \\
        &  = E(\Psi)+\frac{1}{2}\int_{\mathbb{R}^d} \rho_1(\Psi)|\psi_1|^2+\rho_2(\Psi)|\psi_2|^2 \mathrm{d}\mathbf{x}. 
    \end{aligned}
\] 

To compute the ground state of BEC no matter in the single-component or two-component case, various numerical methods have been proposed, based on either energy minimization or eigenvalue characterization, including Sobolev gradient methods, Riemannian optimization techniques, and normalized gradient flows\cite{ref_ROM,Alt,ref_Antoine,ref_Q.tang,ref_I.D.,ref_W.Liu.}.
Among these, the most popular and widely influential algorithm framework especially in physics literature is gradient flow with various realization. And one of the most important methods is the gradient flow with discrete normalization (GFDN), sometimes also called discrete normalized gradient flow (DNGF). The discrete normalized gradient flow with semi-implicit discretization (GFSI) has emerged as a practical approach due to its inherent energy dissipation and mass conservation properties. 

We review the following continuous normalized gradient flow (CNGF), whose various temporal discretizations have been widely employed to solve the ground state problem. By introducing an artificial time variable $t$:
\begin{equation}\label{CNGF}
   \begin{aligned}
    & \frac{\partial}{\partial t} \begin{pmatrix}
        \psi_1(\mathbf{x},t) \\ \psi_2(\mathbf{x},t)
    \end{pmatrix} = \begin{pmatrix}
        -\mathcal{H}_{\psi_1}+\mu_{\Psi}(t) & -\beta \\ -\beta & -\mathcal{H}_{\psi_2}+\mu_{\Psi}(t)
    \end{pmatrix} \begin{pmatrix}
        \psi_1(\mathbf{x},t) \\ \psi_2(\mathbf{x},t)
    \end{pmatrix}, \\
    & \Psi(\mathbf{x},0) = \Psi_0(\mathbf{x}) \in \mathcal{M}, \quad \mathbf{x} \in \mathbb{R}^d, \quad t\geq 0,
\end{aligned} 
\end{equation}
where $\mu_{\Psi}(t) := \mu(\Psi(\cdot,t))$. As we know, the solution $\Psi$ to this CNGF is normalization conserved (or called total mass conserved). The GFSI algorithm mentioned above is obtained by applying the time-splitting method and a semi-implicit discretization to the CNGF.

Theoretical studies of the GFSI have made significant progress in the simpler single-component case\cite{ref_Can,ref_Z.Chen,ref_arXiv}. Specifically, early work in \cite{ref_DNGF} established energy decay for linear systems, while subsequent studies such as \cite{ref_Faou} and \cite{ref_Henning} derived local convergence results for certain nonlinear regimes. More recently, \cite{ref_Zixu} provided a proof of energy dissipation under general conditions.

In contrast, the situation becomes significantly more complicated for multicomponent BECs, which model systems composed of multiple interacting condensates. The inter-component coupling introduces additional computational and theoretical challenges, dramatically altering the energy landscape and requiring novel analytical tools. Although various numerical methods have been proposed, including gradient flows\cite{ref_multi,ref_bao_cai,ref_Kinet}, Newton-type algorithms \cite{ref_Cal,ref_Huang}, alternating minimization \cite{ref_Huang}, Riemannian optimization methods \cite{ref_ROM}, and others \cite{ref_Kaz,ref-Kre,ref_Lin,ref_mish,ref_sch}, rigorous theoretical analysis remains limited compared to the single-component case. Existing studies often rely on numerical heuristics or offer only partial theoretical guarantees. In this work, we take a first step toward addressing this gap by rigorously proving the energy dissipation and global convergence of the GFSI method for two-component BECs with Josephson junction coupling.

 We remark that while energy dissipation is not a strictly necessary condition for the numerical scheme to converge--as observed empirically in Figure \ref{figures case3} where a non-monotone energy path still achieves convergence under larger step sizes--establishing an energy-dissipative property remains of critical importance both theoretically and physically. From a physical perspective, the energy decay ensures that the numerical scheme inherits the  energy-diminishing property of continuous gradient flow, preventing unphysical energy growth, while rigorous proof remains an open challenge. On the other hand, this work connects GFSI with state-of-the-art Riemannian optimization methods for constrained optimization problems. With this theoretical bridge established, the convergence theory, both global and local, for this method, which arises from the numerical discretization of imaginary-time evolution in physics, is now fully resolved.

The paper is organised as follows. Section 2 introduces some notations, the mathematical formulation and basic properties. In section 3, we review the famous GFSI to Gross-Pitaevskii energy functional and provide a rigorous proof of energy dissipation. In section 4, we provide the rigorous proof of the global convergence and further establish a sharp local linear convergence rate. Section 5 is devoted to numerical tests to validate the theoretical findings. Finally, some conclusions are drawn in Section 6.

\section{The constrained energy minimization and basic properties}

In this section, we introduce some basic notations, the mathematical formulation and the basic properties of these mathematical  formulations.

It is well known that if the trapping potential satisfies the confining condition, the ground states decay exponentially fast when $|\mathbf{x}| \to \infty$. So we can truncate the whole space $\mathbb{R}^d$ into a bounded domain $\Omega$ with homogeneous Dirichlet boundary conditions and assume that the boundary is Lipschitz continuous. Practically, throughout this paper we always make assumptions as follows.
\begin{enumerate}
    \item[(1)] \textbf{A1}: The external potentials satisfy \( V_i \in   L^\infty(\Omega) \) and 
     $$ V_i(\mathbf{x}) \geq \frac{1+\alpha}{2} \omega^2_i(x^2+y^2) + |\beta|, \quad \text{for almost all}\ \mathbf{x} \in \Omega, $$
        with $d \geq 2$ where $\alpha > 0$ is certain constant, $i=1, 2$. Under this condition, $V_i$ naturally satisfy the confining condition.
    \item[(2)] \textbf{A2}: The interaction matrix $K := \begin{pmatrix}
k_{11} & k_{12} \\
k_{21} & k_{22}
\end{pmatrix} $ is always symmetric and either positive definite or every entries all non-negative.
\end{enumerate}

\subsection{Notations}

To begin with, we introduce the notations and definitions.

On the bounded domain $\Omega$, we equip the Hilbert spaces $L^2(\Omega)$ and $H^1_0(\Omega)$ with the following real inner products:
\[ 
(u,v)_{L^2}:=\mathrm{Re}\int_\Omega u \overline{v}\mathrm{d}\mathbf{x} \quad\mathrm{and}\quad
(u,v)_{H_0^1}:=\mathrm{Re}\int_\Omega\nabla u\cdot\nabla\overline{v}\mathrm{d}\mathbf{x}.
\]
which also naturally induce the corresponding norms 
\[
\|\cdot\|_{L^2}^2=(\cdot,\cdot)_{L^2} \quad\mathrm{and}\quad 
\|\cdot\|_{H_0^1}^2=(\cdot,\cdot)_{H_0^1}.
\]
Throughout this paper, our discussions are totally based on vector form, hence for any $\mathbf{u}=(u_1, u_2)^{\top}, \mathbf{v}=(v_1,v_2)^{\top} \in [H^1_0(\Omega)]^2$, which is also in $[L^2(\Omega)]^2$, we define the corresponding bilinear form as:
\[
(\mathbf{u},\mathbf{v})_{L^2} := \sum_{i=1,2} (u_i, v_i)_{L^2} ,\quad 
(\mathbf{u},\mathbf{v})_{H_0^1} := \sum_{i=1,2} (u_i, v_i)_{H_0^1}, 
\]
It is standard to verify the definitions above are actually inner products with respect to Lebesgue space $[L^2(\Omega)]^2$ and Sobolev space $[H^1_0(\Omega)]^2$. The corresponding norms are given by
\[
   \|\mathbf{u}\|^2_{L^2} = (\mathbf{u},\mathbf{u})_{L^2}, \quad
     \|\mathbf{u}\|^2_{H^1_0} = (\mathbf{u},\mathbf{u})_{H^1_0}.
\]
To simplify the notation in what follows, we adopt the following convention:
\begin{equation*}
\|u\|_p = \|u\|_{L^p(\Omega)}, \quad
\|\mathbf{u}\|_p = \|\mathbf{u}\|_{[L^p(\Omega)]^2}
= \left( \|u_1\|_p^p + \|u_2\|_p^p \right)^{1/p},
\end{equation*}
for a scalar function $u$ and components $u_i$ in $L^p(\Omega)$, where no ambiguity arises.
After truncation and notations given above, we can rewrite the energy functional as:
\begin{equation}\label{total energy}
\begin{aligned}
E(\Psi) =\sum_{i=1,2} & \int_{\Omega}\frac{1}{2}|\nabla\psi_i|^2+V_i(\mathbf{x})|\psi_i|^2+\frac{1}{2}\rho_i(\Psi)|\psi_i|^2 - \omega_i\overline{\psi_i}L_z\psi_i \mathrm{d}\mathbf{x} \\
& + 2\beta \int_{\Omega}\mathrm{Re}(\psi_1\overline{\psi_2})\mathrm{d}\mathbf{x}.
\end{aligned}
\end{equation}
and ground states (\ref{original mini problem}):
\begin{equation}\label{mini problem}
\Psi_g (\mathbf{x}) := \arg\min_{\Psi \in \mathcal{M}} E(\Psi),
\end{equation}
where $\mathcal{M} := \left\{ \Psi \in [H^1_0(\Omega)]^2 \mid \|\Psi\|_{L^2} =1, E(\Psi) < \infty \right\} $.
Assumptions about the external potentials \textbf{A1} and the interaction matrix \textbf{A2} ensure the existence of a ground state. 
\begin{lemma}\label{existence of ground state}
     There exists $\Psi_g \in \mathcal{M}$, such that $\Psi_g$ is a global minimizer of the constrained minimization problem (\ref{mini problem}).
\end{lemma}
\begin{proof}
    The proof is standard and can be found in Appendix \ref{proof of EXT}.
\end{proof}

We define the Hamiltonian operator $\mathcal{H}_{\psi_i}: H^1_0(\Omega) \to H^{-1}$ for $\psi_i \in H^1_0(\Omega)$, $u, v \in H^1_0(\Omega)$, $i=1,2$ as:
\begin{equation}
\begin{aligned}
   \langle \mathcal{H}_{\psi_i}u,v \rangle
    &= \mathrm{Re}\int_{\Omega}\frac{1}{2}\nabla u \cdot \nabla \overline{v} +V(\mathbf{x})u\overline{v} + \rho_i(\Psi)u\overline{v} -\omega_i\overline{v}L_zu \mathrm{d}\mathbf{x} \\
    &= \frac{1}{2}(u,v)_{H^1_0} + \left((V(\mathbf{x})+\rho_i(\Psi) -\omega_iL_z)u,v\right)_{L^2},
\end{aligned}
\end{equation}
where $H^{-1}=(H^1_0(\Omega))^*$ denotes the corresponding dual space with canonical duality pairing $\langle \cdot,\cdot \rangle$. Then based on $\mathcal{H}_{\psi_i}$, a linear bounded operator $\mathcal{H}_{\Psi}: [H^1_0(\Omega)]^2 \to H^{*}$, where $H^{*}$ is the dual space of $[H^1_0(\Omega)]^2$, can be given by additive structure:
\begin{equation}\label{additive structure}
 \langle \mathcal{H}_{\Psi}\mathbf{u},\mathbf{v} \rangle :=\sum_{i=1,2} \langle \mathcal{H}_{\psi_i}u_i,v_i \rangle  + \beta(u_1,v_2)_{L^2} + \beta(u_2,v_1)_{L^2},
\end{equation}
for $\mathbf{u},\mathbf{v} \in [H^1_0(\Omega)]^2$. Note that $\Delta u$ can be viewed as $-(\nabla u)\cdot \nabla $ when boundary value equals to zero (in $H_0^1(\Omega)$) in the weak sense, we rewrite the Hamiltonian operator $\mathcal{H}_{\psi_i}$ as 
\[
    \mathcal{H}_{\psi_i}=-\frac{1}{2}\Delta + V_i(\mathbf{x}) + \rho_i(\Psi) - \omega_iL_z.
\]
This expression is precisely identical to the definition given in \eqref{lageq}, which also implies that the operator $\mathcal{H}_{\Psi}$ acts like a matrix:
\begin{equation}\label{p-frame operator}
    \mathcal{H}_{\Psi}\mathbf{u} =
\begin{pmatrix}
\mathcal{H}_{\psi_1} & \beta \\
\beta & \mathcal{H}_{\psi_2}
\end{pmatrix} \mathbf{u} = 
\begin{pmatrix}
    \mathcal{H}_{\psi_1}u_1 + \beta u_2 \\
    \mathcal{H}_{\psi_2}u_2 + \beta u_1
\end{pmatrix}.
\end{equation}
Moreover, we introduce a bilinear form $\mathcal{A}_{\Psi}^a(\cdot,\cdot): [H^1_0(\Omega)]^2\times [H^1_0(\Omega)]^2 \to \mathbb{R}$ according to operator $\mathcal{H}_{\Psi}$ ($\Psi \in [H^1_0(\Omega)]^2$ and $a>0$) as:
\begin{equation}\label{bilinear form}
    \mathcal{A}_{\Psi}^a(\mathbf{u},\mathbf{v}) = 
 \langle (\mathcal{I}+a\mathcal{H}_{\Psi})\mathbf{u},\mathbf{v} \rangle= (\mathbf{u},\mathbf{v})_{L^2} + a \langle \mathcal{H}_{\Psi}\mathbf{u},\mathbf{v} \rangle,
\end{equation}
where $\mathcal{I}: [H^1_0(\Omega)]^2 \to H^{*}$ is canonical identification and we denote the identity operator by $I$ to avoid confusion.  Hereinafter, we denote by $C$ a generic constant that may depend on $\Omega$, $d$, $\alpha$, and $\|V\|_{L^{\infty}}$, but is independent of $k_m:=\max\limits_{i,j=1,2}\left\{|k_{ij}|\right\}$. This includes constants arising from Sobolev inequalities.

\subsection[Properties of H\_Psi and A\_Psi\_a(.,.)]{%
  Properties of $\mathcal{H}_{\Psi}$ and $\mathcal{A}_{\Psi}^a(\cdot,\cdot)$%
}

Throughout this paper, we will frequently use the operator $\mathcal{H}_{\Psi}$ and the bilinear form $\mathcal{A}_{\Psi}^a(\cdot,\cdot)$, hence we have the follwing lemma associated with properties of $\mathcal{H}_{\Psi}$ and $\mathcal{A}_{\Psi}^a(\cdot,\cdot)$. 

Before this, we present an inequality that will be frequently used throughout this paper:
\begin{equation}\label{G-N inequality}
	\|\mathbf{u}\|_{4} \leq C\|\mathbf{u}\|_{L^2}^{1-\frac{d}{4}} \|\nabla \mathbf{u}\|_{L^2}^{\frac{d}{4}} = C\|\mathbf{u}\|_{L^2}^{1-\frac{d}{4}} \|\mathbf{u}\|_{H^1_0}^{\frac{d}{4}}.
\end{equation}

\begin{lemma}\label{some props}
    Suppose  \textbf{A1} and \textbf{A2} are satisfied, for any  given $\Psi \in \mathcal{M}$ and any $\mathbf{u},\mathbf{v} \in [H^1_0(\Omega)]^2$, we have the following properties:
   \begin{enumerate}
	\item[(i)]  The linear part $\mathcal{H}_{\mathbf{0}}$ of $\mathcal{H}_\Psi$ satisfies coercivity, i.e.,
	\begin{equation}\label{H0}
     \langle \mathcal{H}_{\mathbf{0}}\mathbf{u},\mathbf{u} \rangle \geq C_0\|\mathbf{u}\|_{H^1_0}^2,
	\end{equation}
	where $C_0 = \frac{\alpha}{2(1+\alpha)}$.
	\item[(ii)] The operator $\mathcal{H}_\Psi$ satisfies continuity, i.e.,
	\begin{equation}\label{bound of operator}
   |\langle \mathcal{H}_{\Psi}\mathbf{u},\mathbf{v} \rangle| \leq C_{\Psi}\|\mathbf{u}\|_{H^1_0} \|\mathbf{v}\|_{H^1_0},
	\end{equation}
	where $C_{\Psi}=Ck_m\|\Psi\|^{\frac{d}{2}}_{H^1_0}$.
	\item[(iii)] The following lower bound estimate holds
	\begin{equation}\label{operator boundede below}
	\langle \mathcal{H}_{\Psi}\mathbf{u},\mathbf{u} \rangle \geq \frac{C_0}{2} \|\mathbf{u}\|^2_{H^1_0} - C^{\frac{4}{4-d}}_{\Psi}\|\mathbf{u}\|^2_{L^2}.
	\end{equation}
\end{enumerate}
\end{lemma}

\begin{proof}
    Firstly, by Young's inequality, for any $\epsilon > 0$, $i=1,2$ we have \[
    \begin{aligned}
        |(\omega_i L_z u_i,u_i)_{L^2}| 
        & \leq \int_{\Omega}|\omega_i||u_i||(y\partial_x-x\partial_y)u_i| \mathrm{d}\mathbf{x} \\
        & \leq \int_{\Omega} \left[ \frac{\epsilon \omega_i^2}{2}(x^2+y^2)|u_i|^2 + \frac{1}{2\epsilon}(|\partial_xu_i|^2 + |\partial_yu_i|^2) \right] \mathrm{d}\mathbf{x} \\
        & \leq \frac{\epsilon \omega_i^2}{2} \int_{\Omega}  (x^2+y^2)|u_i|^2 \mathrm{d}\mathbf{x} + \frac{1}{2\epsilon} \|u_i\|_{H_0^1}^2 \\
        & = \frac{1+\alpha}{2}\omega_i^2 \int_{\Omega}  (x^2+y^2)|u_i|^2 \mathrm{d}\mathbf{x} + \frac{1}{2(1+\alpha)} \|u_i\|_{H_0^1}^2,
    \end{aligned}
    \]
    in which $\epsilon = 1+\alpha$ is taken in the above. And according to the assumption of external potential, we obtain that:
    \[
        (V_i(\mathbf{x})u_i,u_i)_{L^2} = \int_{\Omega}V_i(\mathbf{x})|u_i|^2 \mathrm{d}\mathbf{x} \geq \frac{1+\alpha}{2}\omega_i^2 \int_{\Omega}  (x^2+y^2)|u_i|^2 \mathrm{d}\mathbf{x} + |\beta| \|u_i\|_{H_0^1}^2,
    \]
which implies
\[
\begin{aligned}
  \sum_{i=1,2} (V_i(\mathbf{x})u_i,u_i)_{L^2} - (\omega_i L_z u_i,u_i)_{L^2} & \geq -\frac{\|u_1\|^2_{H_0^1}+\|u_2\|^2_{H_0^1}}{2(1+\alpha)} + |\beta|(\|u_1\|^2_{L^2}+\|u_2\|^2_{L^2}) \\
  & = -\frac{1}{2(1+\alpha)} \|\mathbf{u}\|_{H^1_0}^2 + |\beta|\|\mathbf{u}\|^2_{L^2} \\
  & \geq -\frac{1}{2(1+\alpha)} \|\mathbf{u}\|_{H^1_0}^2 - 2\beta(u_1,u_2)_{L^2}.
\end{aligned}
\]
Hence we obtain that:
\[
\begin{aligned}
     \langle \mathcal{H}_{\mathbf{0}}\mathbf{u},\mathbf{u} \rangle & = \sum_{i=1,2} \left[ \frac{1}{2}\|u_i\|^2_{H_0^1} + (V_i(\mathbf{x})u_i,u_i)_{L^2} - (\omega_i L_z u_i,u_i)_{L^2} \right] + 2\beta(u_1,u_2)_{L^2}\\
     & \geq \frac{1}{2}\|\mathbf{u}\|^2_{H^1_0} - \frac{1}{2(1+\alpha)}\|\mathbf{u}\|^2_{H^1_0} = C_0\|\mathbf{u}\|_{H^1_0}^2.
\end{aligned}
\]

Then (\ref{H0}) has been proved, as for the proof of (\ref{bound of operator}), it is standard and can be directly verified as the same way in \cite{ref_Zixu}. 

For \eqref{operator boundede below}, we derive that
\[  \begin{aligned}
	\int_{\Omega}|\rho_i(\Psi)||u_i|^2\mathrm{d}\mathbf{x} 
	& \leq \sum_{j=1,2} \int_{\Omega}|k_{ij}||\psi_j|^2|u_i|^2\mathrm{d}\mathbf{x}\leq \sum_{j=1,2}|k_{ij}|\|\psi_j\|_4^2 \|u_i\|_4^2 \\
	& \leq 
	k_m \|\Psi\|_4^2 \|u_i\|_4^2\leq Ck_m\|\Psi\|_{H^1_0}^{\frac{d}{2}}\|u_i\|_2^{2-\frac{d}{2}}\|u_i\|_{H^1_0}^{\frac{d}{2}}.
\end{aligned}  \]
Combined with Young's inequality, we have 
\begin{align*}
	\sum_{i=1,2}&\int_{\Omega}|\rho_i(\Psi)||u_i|^2\mathrm{d}\mathbf{x}\leq Ck_m\|\Psi\|_{H^1_0}^{\frac{d}{2}}\|\mathbf{u}\|^{2-\frac{d}{2}}_{L^2}\|\mathbf{u}\|^{\frac{d}{2}}_{H^1_0}\\& \leq \epsilon^{-\frac{d}{4-d}}Ck_m\|\Psi\|_{H^1_0}^{\frac{d}{2}} \|\mathbf{u}\|^2_{L^2} + \epsilon Ck_m\|\Psi\|_{H^1_0}^{\frac{d}{2}}\|\mathbf{u}\|^2_{H^1_0}\\
	&=C\left(k^2_m\|\Psi\|_{H^1_0}^{d}\right)^{\frac{2}{4-d}}\|\mathbf{u}\|^2_{L^2}+\frac{C_0}{2}\|\mathbf{u}\|^2_{H^1_0},\quad \text{when}\quad \epsilon=\frac{C_0}{2Ck_m\|\Psi\|_{H^1_0}^{\frac{d}{2}}}.
\end{align*}
This means that:
\begin{align*}
	\langle \mathcal{H}_{\Psi}\mathbf{u},\mathbf{u} \rangle\geq	\langle \mathcal{H}_{\bm{0}}\mathbf{u},\mathbf{u} \rangle-\sum_{i=1,2}\int_{\Omega}|\rho_i(\Psi)||u_i|^2\mathrm{d}\mathbf{x}  \geq \frac{C_0}{2} \|\mathbf{u}\|^2_{H^1_0} -C^{\frac{4}{4-d}}_{\Psi}\|\mathbf{u}\|^2_{L^2}.
\end{align*}
\end{proof}

In addition, concerned with $\mathcal{A}_{\Psi}^a(\cdot,\cdot)$ we have properties as follows.
\begin{lemma}\label{Property-A}
    Suppose \textbf{A1} and \textbf{A2} are satisfied, for any  given $\Psi \in \mathcal{M}$, the bilinear form $\mathcal{A}_{\Psi}^a(\cdot,\cdot)$ defined in (\ref{bilinear form}) satisfies the following properties:
   \begin{enumerate}
    	\item[(i)] $\mathcal{A}_{\Psi}^a(\cdot,\cdot)$ is symmetric and bounded.
    	\item[(ii)] For $0<a \le a_{\Psi}:=1/C^{\frac{4}{4-d}}_{\Psi}$, $\mathcal{A}_{\Psi}^a(\cdot,\cdot)$ is coercive with
    	 \begin{equation}\label{coercivity of bf}
    		\mathcal{A}_{\Psi}^a(\mathbf{u},\mathbf{u}) \geq \frac{aC_0}{2}\|\mathbf{u}\|^2_{H^1_0}, \quad \mathbf{u}\in [H^1_0(\Omega)]^2.
    	\end{equation}
    \end{enumerate}
\textcolor{blue}{That is, for $0 < a \le a_{\Psi}$, $\mathcal{A}_{\Psi}^a(\mathbf{u},\mathbf{u})$ naturally induces an inner product $(\cdot, \cdot)_{\mathcal{A}_{\Psi}^a}$ and the associated norm $\|\cdot\|_{\mathcal{A}_{\Psi}^a}$ on $[H_0^1(\Omega)]^2$.}
\end{lemma}
 \begin{proof}
 	The symmetry is straightforward to verify. As for the boundedness and coerciveness, they follow directly from Lemma \ref{some props} $(ii)$ and $(iii)$, respectively.
 \end{proof}

\begin{remark}
    Apparently, if all entries of $K$ are non-negative, then $\langle \mathcal{H}_{\Psi}\mathbf{u},\mathbf{u} \rangle \geq \langle \mathcal{H}_{\mathbf{0}}\mathbf{u},\mathbf{u} \rangle \geq C_0\|\mathbf{u}\|_{H^1_0}^2$, which means $\mathcal{H}_{\Psi}$ and $\mathcal{A}_{\Psi}^a(\cdot,\cdot)$ are both coercive. But it is not a necessary condition to the positive definiteness of $K$. However, even without assuming that all entries of $K$ are non-negative nor choosing $a$ so technical way like Lemma \ref{Property-A} $(ii)$, we can still assume the operator $\mathcal{H}_{\Psi}$ is coercive in later discussion. That is because, for any $\|\Psi\|_{H^1_0}\le M$, we can always replace $\mathcal{H}_{\Psi}$ by $\langle \mathcal{H}_{\Psi}\cdot,\cdot \rangle + g_{M}(\cdot,\cdot)_{L^2}$ where $g_{M}>0$ is an appropriately chosen constant such that $\langle \mathcal{H}_{\Psi}\cdot,\cdot \rangle + g_{M}(\cdot,\cdot)_{L^2}$ is coercive according to (\ref{operator boundede below}). In the sense of minimization problem, we can view this change corresponds to adding $g_{M}$ to $V(\mathbf{x})$, the constant shift do not change the global minimizer of (\ref{mini problem}).
In a word, the $\mathcal{A}_{\Psi}^a(\cdot,\cdot)$ is always coercive when $0<a\le a_{\Psi}$ is satisfied. And we assume $\mathcal{H}_{\Psi}$ to be coercive with $\langle \mathcal{H}_{\Psi}\mathbf{u},\mathbf{u} \rangle \geq\langle \mathcal{H}_{\bm{0}}\mathbf{u},\mathbf{u} \rangle\geq C_0\|\mathbf{u}\|_{H^1_0}^2$ in the remainder of this paper. While readers need to be aware that the case of all entries of $K$ being non-negative is validly covered, and in a general context it is still a reasonable assumption based on discussions above.
\end{remark}

\section{The reformulation of GFSI algorithm and energy dissipation}

In this section, we review the GFSI algorithm to the coupled Gross-Pitaevskii energy functional and provide a rigorous proof of energy dissipation by providing an equivalent form. Details will be provided later.

In general, in order to discretize the CNGF, we adopt the time-splitting method and semi-implicit time discretization. This discrete normalized gradient flow with semi-implicit time discretization algorithm (GFSI) has become the most widely used semi-discretization algorithm for computing the ground state of BEC on account of its implicit treatment of the linear terms and stable numerical simulation.

\subsection{The GFSI algorithm}

The\; gradient\; flow\; with\; discrete\; normalization 
(GFDN) is obtained by applying the time-splitting method to CNGF(\ref{CNGF}). For a fixed time step size $\tau$, denote the time steps as $t_n=n\tau (\tau>0)$ for $n=0,1,\cdots$, we have the GFDN as 
\begin{equation}\label{GFDN}
    \begin{aligned}
        & \frac{\partial}{\partial t} \begin{pmatrix}
            \psi_1 \\ \psi_2
        \end{pmatrix} = \begin{pmatrix}
            -\mathcal{H}_{\psi_1} & - \beta \\
            - \beta & -\mathcal{H}_{\psi_2}
        \end{pmatrix} \begin{pmatrix}
            \psi_1 \\ \psi_2
        \end{pmatrix} \quad  \psi_1 =0 \quad \text{on} \quad \partial\Omega,\\ &\Psi(\mathbf{x},t_{n+1}):=\Psi(\mathbf{x},t^+_{n+1})=\Psi(\mathbf{x},t^{-}_{n+1})/\|\Psi(\mathbf{x},t^{-}_{n+1})\|_{L^2},
    \end{aligned}
\end{equation}
with initial value $\Psi(\mathbf{x},0) = \Psi_0(\mathbf{x}) \in \mathcal{M}$.
Furthermore, the GFSI algorithm is obtained by applying a semi-implicit discretization in time variable to GFDN(\ref{GFDN}). Let $\Psi^n=(\psi_1^n,\psi_2^n)^{\top}=(\psi_1(\cdot,t_n),\psi_2(\cdot,t_n))^{\top}$ be the numerical approximation, we have the GFSI reads (for $ n \geq 0$ and $i=1,2.$)
\[ \begin{aligned}
    &\frac{\widetilde{\psi}^{n+1}_1 - \psi_1^n}{\tau} = -\mathcal{H}_{\psi_1^n}\widetilde{\psi}^{n+1}_1 -\beta \widetilde{\psi}^{n+1}_2, \quad 
    \psi_1^{n+1}=\frac{\widetilde{\psi}^{n+1}_1}{\|\widetilde{\Psi}^{n+1}\|_{L^2}};\\
    &\frac{\widetilde{\psi}^{n+1}_2 - \psi_2^n}{\tau} = -\mathcal{H}_{\psi_2^n}\widetilde{\psi}^{n+1}_2 -\beta \widetilde{\psi}^{n+1}_1 , \quad 
    \psi_2^{n+1}=\frac{\widetilde{\psi}^{n+1}_2}{\|\widetilde{\Psi}^{n+1}\|_{L^2}}, \quad \widetilde{\psi}^{n+1}_i=0 \quad \text{on} \quad \partial\Omega,
\end{aligned} \]
in which $\|\widetilde{\Psi}\|^2_{L^2} = \|\widetilde{\psi}^{n+1}_1\|^2_{L^2} + \|\widetilde{\psi}^{n+1}_1\|^2_{L^2}$, $(\psi^0_1,\psi^0_2)^{\top}=\Psi_0\in \mathcal{M}$.

By denoting $\widetilde{\Psi}^{n+1}=\left( \widetilde{\psi}^{n+1}_1,\widetilde{\psi}^{n+1}_2 \right)^{\top}$ and Hamiltonian operator (\ref{p-frame operator}), we can rewrite GFSI as:
\begin{equation}\label{original GFSI}
    \frac{\widetilde{\Psi}^{n+1} - \Psi^n}{\tau} = -\mathcal{H}_{\Psi^n}\widetilde{\Psi}^{n+1}, \hspace{0.5em}     \Psi^{n+1}=\widetilde{\Psi}^{n+1}/\|\widetilde{\Psi}^{n+1}\|_{L^2}, \hspace{0.5em} \widetilde{\Psi}^{n+1}=0 \hspace{0.5em} \text{on} \hspace{0.5em} \partial\Omega,
\end{equation}
where $\Psi^0=\Psi_0 \in \mathcal{M}$.
\begin{remark}\label{L-M lemma}
    Noticing (\ref{bound of operator}) and the coercivity of $\mathcal{A}_{\Psi}^a (0<a\leq a_{\Psi})$ we stated in (\ref{coercivity of bf}), by Lax-Milgram Lemma, for any $\mathbf{w}\in H^*$, there exists a unique $\mathbf{u}\in [H^1_0(\Omega)]^2$, s.t.
    \[
        \mathcal{A}_{\Psi}^a(\mathbf{u},\mathbf{v})= \langle \mathbf{w},\mathbf{v} \rangle, \quad \text{for any} \quad \mathbf{v}\in [H^1_0(\Omega)]^2,
    \]
    where $\mathbf{u}$ can be denoted by $(\mathcal{I}+a\mathcal{H}_{\Psi})^{-1}\mathbf{w} \in [H^1_0(\Omega)]^2$.
    And also $(\mathcal{I}+a\mathcal{H}_{\Psi})^{-1}\mathbf{w}$ is the unique weak solution to the equation $(\mathcal{I}+a\mathcal{H}_{\Psi}) \mathbf{u}=\mathbf{w}$ with homogeneous boundary condition.
\end{remark}

Note that (\ref{original GFSI}) is equivalent to $(\mathcal{I}+\tau \mathcal{H}_{\Psi^n})\widetilde{\Psi}^{n+1}=\Psi^n$ with $\widetilde{\Psi}^{n+1}=0$ on $\partial \Omega$, then by (\ref{bound of operator}) and (\ref{coercivity of bf}) as well as Remark \ref{L-M lemma}, we have 
\[
    \widetilde{\Psi}^{n+1}=(\mathcal{I}+\tau \mathcal{H}_{\Psi^n})^{-1}\Psi^n, \qquad n\geq 0,
\]
as the weak solution of (\ref{original GFSI}) is well-defined for any $0<\tau \leq a_{\Psi^n}:=1/C^{\frac{4}{4-d}}_{\Psi^n}$, which leads to the existence of $\widetilde{\Psi}^{n+1} \in [H^1_0(\Omega)]^2$. Then by (\ref{original GFSI}) we know that
\[
    \mathcal{H}_{\Psi^n}\widetilde{\Psi}^{n+1} =-\frac{\widetilde{\Psi}^{n+1} - \Psi^n}{\tau},
\]
as $\Psi^n$ is also in $[H^1_0(\Omega)]^2$, we claim  $\mathcal{H}_{\Psi^n}\widetilde{\Psi}^{n+1} \in [H^1_0(\Omega)]^2$. This conclusion is non-trivial as notice that for any $\Psi^n \in [H^1_0(\Omega)]^2$,the operator $\mathcal{H}_{\Psi^n} : [H^1_0(\Omega)]^2 \to H^*$ implies that $\mathcal{H}_{\Psi^n}\mathbf{u}$ usually be in $H^*$ with a general $\mathbf{u} \in [H^1_0(\Omega)]^2$ ($[H^1_0(\Omega)]^2 \subset H^*$ is canonical inclusion).

\subsection{The energy dissipation}

In this subsection, we provide following theorem concerned with the energy dissipation of GFSI (\ref{original GFSI}) and some important and necessary preparatory work to prove the theorem.

\begin{theorem}[{\bf Energy dissipation}]\label{THM}
 With \textbf{A1} and \textbf{A2} satisfied, for any given initial $\Psi_0 \in \mathcal{M}$, there exists $\tau_0 >0$, which depends on $E(\Psi_0)$ and $k_m$, and can be given by 
   \begin{equation}\label{tau0}
       \tau_0 = \min\left\{ \frac{2}{\widetilde{C}_{E^0}}, \frac{1}{4E(\Psi^0)} \right\}, \hspace{0.5em} \widetilde{C}_{E^0}= \left(2Ck_mC_{E^0}^{\frac{d}{2}} \right)^{\frac{4}{4-d}}, \hspace{0.5em} C_{E^0}=C\sqrt{E(\Psi^0)}.
   \end{equation}
    Then for any $0<\tau<\tau_0$, the sequence $\{\Psi^n\}_{n \in \mathbb{N}}$ generated by (\ref{original GFSI}) is well-defined and energy dissipative, i.e.
    \begin{equation}\label{energy dis}
        E(\Psi^{n+1})-E(\Psi^n)\leq - \frac{C_0}{2}\|\Psi^{n+1}-\Psi^n\|_{H^1_0}^2, \quad n \geq 0.
    \end{equation} 
\end{theorem}
The proof is slightly tedious, we start with some important preparation.\\
Representing (\ref{original GFSI}) in an equivalent form by plugging $\widetilde{\Psi}^{n+1}=\|\widetilde{\Psi}^{n+1}\|_{L^2}\Psi^{n+1}$ into (\ref{original GFSI}) obtains that
\begin{equation}\label{GFSI}
    \frac{\Psi^{n+1} - \Psi^n}{\tau} = -\mathcal{H}_{\Psi^n}\Psi^{n+1} + \lambda^{n+1}\Psi^n, \quad \lambda^{n+1}=\frac{1-\|\widetilde{\Psi}^{n+1}\|_{L^2}}{\tau \|\widetilde{\Psi}^{n+1}\|_{L^2}},
\end{equation}
where $\Psi^n=0$ on $\partial \Omega$, $n \geq 0$.

For the convenience of later expression, we always denote 
\[
\begin{aligned}
    \mathbf{\rho}^n &:= \begin{pmatrix} |\psi_1^n|^2 \\ |\psi_2^n|^2 \end{pmatrix}, \quad
    \mathbf{\rho}^{n+1} := \begin{pmatrix} |\psi_1^{n+1}|^2 \\ |\psi_2^{n+1}|^2 \end{pmatrix}, \quad
    \hat{\mathbf{\rho}}^{n+1} := \begin{pmatrix} |\psi_1^{n+1} - \psi_1^n|^2 \\ |\psi_2^{n+1} - \psi_2^n|^2 \end{pmatrix},
\end{aligned}
\]
in specific discussion.
\begin{lemma}\label{LMA 1st-solu L-norm}
    Suppose \textbf{A1} and \textbf{A2} are satisfied, for any given initial $\Psi_0 \in \mathcal{M}$, the norm of $\widetilde{\Psi}^{n+1}$ generated by (\ref{original GFSI}) in $L^2$ is less than or equal to $1$, i.e.
    \begin{equation}\label{1st-solu L-norm}
        \|\widetilde{\Psi}^{n+1}\|_{L^2} \leq 1, \quad n \geq 0.
    \end{equation}
\end{lemma}
\begin{proof}
    Acting with $\tau \widetilde{\Psi}^{n+1}$ on both sides of (\ref{original GFSI}), we obtain
\[
    \|\widetilde{\Psi}^{n+1}\|^2_{L^2}=(\Psi^n,\widetilde{\Psi}^{n+1})_{L^2} - \tau \langle \mathcal{H}_{\Psi^n}\widetilde{\Psi}^{n+1},\widetilde{\Psi}^{n+1} \rangle.
\]
    Note the coercivity of $\mathcal{H}_{\Psi_n}$ and $\|\Psi^n\|_{L^2}=1$, we have
    \[
        \|\widetilde{\Psi}^{n+1}\|^2_{L^2} \leq (\Psi^n,\widetilde{\Psi}^{n+1})_{L^2} \leq \|\Psi^n\|_{L^2} \|\widetilde{\Psi}^{n+1}\|_{L^2},
    \]
    which implies $\|\widetilde{\Psi}^{n+1}\|_{L^2} \leq 1$.
\end{proof}
\begin{corollary}\label{Corollary}
    Suppose \textbf{A1} and \textbf{A2} are satisfied, for any given initial $\Psi_0 \in \mathcal{M}$, the Lagrange multiplier  $\lambda^{n+1}$ generated by the reformulated GFSI algorithm (\ref{GFSI}) is always non-negative, i.e.
    $$\lambda^{n+1} \geq 0, \quad \text{for} \quad n \geq 0. $$
\end{corollary}
\begin{proof}
    The corollary is trivial to verify associated (\ref{GFSI}) with Lemma \ref{LMA 1st-solu L-norm} .
\end{proof}

The following series of Lemmas are intended to show the norm of $\Psi^{n+1}$ in $[H^1_0(\Omega)]^2$ can be upper bounded by an expression related to $E(\Psi^n)$.

\begin{lemma}\label{pre-lemma 1}
Suppose \textbf{A1} and \textbf{A2} are satisfied and given initial $\Psi_0$ be in $\mathcal{M}$, for $\{\Psi^n\}_{n \in \mathbb{N}}$ generated by (\ref{original GFSI}) and $\widetilde{\Psi}^{n+1}=\|\widetilde{\Psi}^{n+1}\|_{L^2}\Psi^{n+1}$, we have that $$\langle\mathcal{H}_{\Psi^n}\Psi^n,\mathcal{H}_{\Psi^n}\widetilde{\Psi}^{n+1} \rangle \geq 0, \quad \text{for all} \quad n \geq 0.$$
\end{lemma}
\begin{proof}
    Notice that $\Psi^n=\widetilde{\Psi}^{n+1} + \tau \mathcal{H}_{\Psi^n}\widetilde{\Psi}^{n+1}$ by (\ref{original GFSI}), we obtain
\[  \begin{aligned}
    \langle \mathcal{H}_{\Psi^n}\Psi^n,\mathcal{H}_{\Psi^n}\widetilde{\Psi}^{n+1} \rangle & = \langle \mathcal{H}_{\Psi^n}\widetilde{\Psi}^{n+1} + \mathcal{H}_{\Psi^n}(\tau \mathcal{H}_{\Psi^n}\widetilde{\Psi}^{n+1}),\mathcal{H}_{\Psi^n}\widetilde{\Psi}^{n+1} \rangle\\
    & = \langle \mathcal{H}_{\Psi^n}\widetilde{\Psi}^{n+1},\mathcal{H}_{\Psi^n}\widetilde{\Psi}^{n+1} \rangle + \tau \langle \mathcal{H}_{\Psi^n}(\mathcal{H}_{\Psi^n}\widetilde{\Psi}^{n+1}),\mathcal{H}_{\Psi^n}\widetilde{\Psi}^{n+1} \rangle\\
    & \geq \|\mathcal{H}_{\Psi^n}\widetilde{\Psi}^{n+1}\|^2_{L^2} + C_0\tau\|\mathcal{H}_{\Psi^n}\widetilde{\Psi}^{n+1}\|^2_{H^1_0} \geq 0.
    \end{aligned} \]
    In last line we use the fact that $\mathcal{H}_{\Psi^n}\widetilde{\Psi}^{n+1} \in [H^1_0(\Omega)]^2$ and $\mathcal{H}_{\Psi^n}$ is coercive.
\end{proof}

\begin{lemma}\label{pre-lemma 2}
    Suppose \textbf{A1} and \textbf{A2} are satisfied, for any given initial $\Psi_0 \in \mathcal{M}$, we have the norm of $\widetilde{\Psi}^{n+1}$ in $[H^1_0(\Omega)]^2$ can be upper bounded by following inequality
    \begin{equation}\label{1st-solu H-norm}
         \|\widetilde{\Psi}^{n+1}\|_{H^1_0} \leq \sqrt{\frac{2}{C_0}E(\Psi^n)}, \quad n \geq 0.
    \end{equation}
    where $C_0$ is the coercive constant of $\mathcal{H}_{\Psi^n}$ as previously stated.
\end{lemma}
\begin{proof}
    Firstly acting with $\tau \mathcal{H}_{\Psi^n}\widetilde{\Psi}^{n+1}$ on both sides of (\ref{original GFSI}):
    \[\begin{aligned}
(\widetilde{\Psi}^{n+1},\mathcal{H}_{\Psi^n}\widetilde{\Psi}^{n+1})_{L^2} & = (\Psi^n,\mathcal{H}_{\Psi^n}\widetilde{\Psi}^{n+1})_{L^2} - \tau (\mathcal{H}_{\Psi^n}\widetilde{\Psi}^{n+1},\mathcal{H}_{\Psi^n}\widetilde{\Psi}^{n+1})_{L^2}\\
 & \leq (\Psi^n,\mathcal{H}_{\Psi^n}\widetilde{\Psi}^{n+1})_{L^2} = \langle \mathcal{H}_{\Psi^n}\Psi^n,\widetilde{\Psi}^{n+1} \rangle,
    \end{aligned} \]
with the symmetry of $\langle \mathcal{H}_{\Psi^n}\cdot,\cdot \rangle$ and $(\cdot,\cdot)_{L^2}$.
    Then notice
\[
(\widetilde{\Psi}^{n+1},\mathcal{H}_{\Psi^n}\widetilde{\Psi}^{n+1})_{L^2} = \langle \mathcal{H}_{\Psi^n}\widetilde{\Psi}^{n+1},\widetilde{\Psi}^{n+1} \rangle \geq C_0\|\widetilde{\Psi}^{n+1}\|^2_{H^1_0},
\]
associated with (\ref{total energy}), (\ref{additive structure}) and (\ref{p-frame operator}) suggests 
\[ \begin{aligned}
    C_0\|\widetilde{\Psi}^{n+1}\|^2_{H^1_0} 
    & \leq \langle \mathcal{H}_{\Psi^n}\Psi^n,\widetilde{\Psi}^{n+1} \rangle  = \langle \mathcal{H}_{\Psi^n}\Psi^n,\Psi^n - \tau \mathcal{H}_{\Psi^n}\widetilde{\Psi}^{n+1} \rangle \\
    & \leq \langle \mathcal{H}_{\Psi^n}\Psi^n,\Psi^n \rangle  = E(\Psi^n) + \frac{1}{2}\int_{\Omega}(\mathbf{\rho}^n)^{\top}K\mathbf{\rho}^n\mathrm{d}\mathbf{x} \\
    & = E(\Psi^n) +(E(\Psi^n)-\langle \mathcal{H}_{\mathbf{0}}\Psi^n,\Psi^n \rangle) \leq 2E(\Psi^n),
\end{aligned}  \]
which is also based on the fact of Lemma \ref{pre-lemma 1}. 
This completes the proof.
\end{proof}
\begin{lemma}\label{LMA 2nd-solu H-norm}
    Suppose assumptions \textbf{A1} and \textbf{A2} both satisfied, for any given initial $\Psi_0 \in \mathcal{M}$, the norm of $\Psi^{n+1}$ in $[H^1_0(\Omega)]^2$ can be upper bounded by choosing appropriate time step size, $n \geq0$, i.e. 
    \begin{equation}\label{2nd-solu H-norm}
    \|\Psi^{n+1}\|_{H^1_0} \leq C_{E^n}, \quad C_{E^n}:=C\sqrt{E(\Psi^n)} \quad \text{for} \quad 0<\tau\leq1/(4E(\Psi^n)).
    \end{equation}
\end{lemma}
\begin{proof}
    By the fact that $\|\Psi^n\|_{L^2}=1$ and $\langle \mathcal{H}_{\Psi^n}\Psi^n,\widetilde{\Psi}^{n+1} \rangle \leq 2E(\Psi^n)$ (which is implied in proof of Lemma \ref{pre-lemma 2}), acting with $\tau \Psi^n$ on both sides of (\ref{original GFSI}) we have:
    \[
        (\widetilde{\Psi}^{n+1},\Psi^n)_{L^2} = (\Psi^n,\Psi^n)_{L^2} - \tau \langle \mathcal{H}_{\Psi^n}\widetilde{\Psi}^{n+1},\Psi^n \rangle \geq 1- 2\tau E(\Psi^n).
    \]
    Then $(\widetilde{\Psi}^{n+1},\Psi^n)_{L^2} \leq \|\widetilde{\Psi}^{n+1}\|_{L^2}\|\Psi^n\|_{L^2} = \|\widetilde{\Psi}^{n+1}\|_{L^2}$ means 
    \[
        \|\widetilde{\Psi}^{n+1}\|_{L^2} \geq 1- 2\tau E(\Psi^n).
    \]
    We might take $\tau \leq1/(4E(\Psi^n))$, then obtain that $\|\widetilde{\Psi}^{n+1}\|_{L^2} \geq 1/2$.
    Then according to (\ref{1st-solu H-norm}), we have:
    \[  \begin{aligned}
        \|\Psi^{n+1}\|_{H^1_0} 
=\|\widetilde{\Psi}^{n+1}\|_{H_0^1}/\|\widetilde{\Psi}^{n+1}\|_{L^2}\leq C\sqrt{E(\Psi^n)}, \quad \text{denoted as} \quad C_{E^n}.
    \end{aligned} \]
\end{proof}

With all the preparatory work from the previous statement, we give the valid proof of energy dissipation in this subsection. To begin with, noting $\|\Psi^{n+1}\|_{L^2}=\|\Psi^n\|_{L^2}=1$ we have a preliminary result that 
\[ \begin{aligned}
    2\lambda^{n+1}(\Psi^n, & \Psi^{n+1}-\Psi^n)_{L^2}
     = \lambda^{n+1}[2(\Psi^n,\Psi^{n+1})_{L^2}-2(\Psi^n,\Psi^n)_{L^2}]\\
    & = \lambda^{n+1}[(\Psi^n,\Psi^n)_{L^2}+(\Psi^{n+1},\Psi^{n+1})_{L^2}-\|\Psi^{n+1}-\Psi^n\|^2_{L^2}-2(\Psi^n,\Psi^n)_{L^2}]\\
    & = -\lambda^{n+1}\|\Psi^{n+1}-\Psi^n\|^2_{L^2}.
\end{aligned} \]
Then acting with $2(\Psi^{n+1}-\Psi^n)$ on both sides of (\ref{GFSI}):
\[\begin{aligned}
    \frac{2}{\tau}\|\Psi^{n+1}-\Psi^n\|^2_{L^2} 
    &= -2\langle \mathcal{H}_{\Psi^n}\Psi^{n+1},\Psi^{n+1}-\Psi^n \rangle+2\lambda^{n+1}(\Psi^n,\Psi^{n+1}-\Psi^n)_{L^2}\\
    & = \langle \mathcal{H}_{\Psi^n}\Psi^n,\Psi^n \rangle - \langle \mathcal{H}_{\Psi^n}\Psi^{n+1},\Psi^{n+1} \rangle \\
    & \quad - \langle \mathcal{H}_{\Psi^n}(\Psi^{n+1}-\Psi^n),(\Psi^{n+1}-\Psi^n) \rangle - \lambda^{n+1}\|\Psi^{n+1}-\Psi^n\|^2_{L^2},
\end{aligned}
\]
by (\ref{H0}) and Corollary \ref{Corollary}  obtain that
\[
\begin{aligned}
    & \quad   \langle \mathcal{H}_{\Psi^n}\Psi^{n+1},\Psi^{n+1} \rangle - \langle \mathcal{H}_{\Psi^n}\Psi^n,\Psi^n \rangle \\
    & = -(\frac{2}{\tau}+\lambda^{n+1})\|\Psi^{n+1}-\Psi^n\|^2_{L^2} - \langle \mathcal{H}_{\Psi^n}(\Psi^{n+1}-\Psi^n),(\Psi^{n+1}-\Psi^n) \rangle \\
    & \leq -\frac{2}{\tau}\|\Psi^{n+1}-\Psi^n\|^2_{L^2} -C_0\|\Psi^{n+1}-\Psi^n\|_{H^1_0}^2.
\end{aligned}
\]

It is an equality that will be used in following proof. Hence, according to all these formulas and inequalities, we are ready to prove Theorem \ref{THM}.
\begin{proof}[{\bf Proof of Theorem \ref{THM}}]
    Firstly, by (\ref{total energy}), (\ref{additive structure}) and notations we defined in the last subsection, we have:
    \[\begin{aligned}
        & E(\Psi^n) = \langle \mathcal{H}_{\Psi^n}\Psi^n,\Psi^n \rangle -\frac{1}{2}\int_{\Omega}(\mathbf{\rho}^n)^{\top}K\mathbf{\rho}^n\mathrm{d}\mathbf{x},\\
        & E(\Psi^{n+1}) = \langle \mathcal{H}_{\Psi^n}\Psi^{n+1},\Psi^{n+1} \rangle + \frac{1}{2}\int_{\Omega}(\mathbf{\rho}^{n+1})^{\top}K\mathbf{\rho}^{n+1}\mathrm{d}\mathbf{x} - \int_{\Omega}(\mathbf{\rho}^n)^{\top}K\mathbf{\rho}^{n+1}\mathrm{d}\mathbf{x}.
    \end{aligned}\]
    Subtract the two formulas obtain that
    \[\begin{aligned}
         &E(\Psi^{n+1}) -E(\Psi^n) 
         =\langle \mathcal{H}_{\Psi^n}\Psi^{n+1},\Psi^{n+1} \rangle - \langle \mathcal{H}_{\Psi^n}\Psi^n,\Psi^n \rangle \\
         &\hspace{6cm}+ \frac{1}{2}\int_{\Omega}(\mathbf{\rho}^{n+1}-\mathbf{\rho}^n)^{\top}K(\mathbf{\rho}^{n+1}-\mathbf{\rho}^n)\mathrm{d}\mathbf{x}\\
        & \leq -\frac{2}{\tau}\|\Psi^{n+1}-\Psi^n\|^2_{L^2} -C_0\|\Psi^{n+1}-\Psi^n\|_{H^1_0}^2 +Ck_m\int_{\Omega}(\mathbf{\rho}^{n+1}-\mathbf{\rho}^n)^{\top}(\mathbf{\rho}^{n+1}-\mathbf{\rho}^n)\mathrm{d}\mathbf{x}\\
        & \leq -\frac{2}{\tau}\|\Psi^{n+1}-\Psi^n\|^2_{L^2} -C_0\|\Psi^{n+1}-\Psi^n\|_{H^1_0}^2 +Ck_m \int_{\Omega}(\mathbf{\rho}^n+\mathbf{\rho}^{n+1})^{\top}\hat{\mathbf{\rho}}^{n+1}\mathrm{d}\mathbf{x}.
    \end{aligned}
    \]
    Another aspect is to notice that by H\"{o}lder inequality, we have
    \[\begin{aligned}
        \int_{\Omega}(\mathbf{\rho}^n+\mathbf{\rho}^{n+1})^{\top}\hat{\mathbf{\rho}}^{n+1}\mathrm{d}\mathbf{x}
        & \leq C\left(\|\Psi^{n+1}\|_4^2+\|\Psi^{n}\|_4^2 \right)\|\Psi^{n+1}-\Psi^n\|_4^2.
    \end{aligned}\]
    Then based on (\ref{G-N inequality}) associated with Young's inequality $ab\leq \epsilon^{\frac{-d}{4-d}}a^{\frac{4}{4-d}} + \epsilon b^{\frac{4}{d}}$, \eqref{2nd-solu H-norm} and $\|\Psi^{n+1}\|_{L^2} = 1$ we have
    \[\begin{aligned}
        Ck_m\int_{\Omega}(\mathbf{\rho}^n+\mathbf{\rho}^{n+1})^{\top}\hat{\mathbf{\rho}}^{n+1}\mathrm{d}\mathbf{x} 
         \leq Ck_mC_{E^n}^{\frac{d}{2}} \left( \epsilon^{\frac{-d}{4-d}}\|\Psi^{n+1}-\Psi^n\|^2_{L^2} + \epsilon \|\Psi^{n+1}-\Psi^n\|_{H^1_0}^2 \right).
    \end{aligned}\]
    We might let
    \[
        \epsilon=\frac{C_0}{2Ck_mC_{E^n}^{\frac{d}{2}}}, \qquad \widetilde{C}_{E^n}= \left(2Ck_mC_{E^n}^{\frac{d}{2}} \right)^{\frac{4}{4-d}},
    \]
    then 
    \begin{align*}
    	E(\Psi^{n+1})-E(\Psi^n)\leq \left(\widetilde{C}_{E^n}-\frac{2}{\tau} \right)\|\Psi^{n+1}-\Psi^n\|^2_{L^2} - \frac{C_0}{2}\|\Psi^{n+1}-\Psi^n\|_{H^1_0}^2.
    \end{align*}
Therefore, choosing time step size $\tau$ such that $\tau \leq\frac{2}{\widetilde{C}_{E^n}}$ with \eqref{2nd-solu H-norm} leads to the energy dissipation, i.e. as long as at $n$-th step 
$$\tau_n = \min\left\{ \frac{2}{\widetilde{C}_{E^n}}, \frac{1}{4E(\Psi^n)} \right\},$$ 
(noting  $\widetilde{C}_{E^n}$ and $4E(\Psi^n)$ decrease as $n$ increases implies this is a reasonable choice). Thus, we have 
\begin{align*}
	E(\Psi^{n+1})-E(\Psi^n)\leq - \frac{C_0}{2}\|\Psi^{n+1}-\Psi^n\|_{H^1_0}^2,\qquad \text{for}\ 0<\tau\leq\tau_0,
\end{align*}
where $\tau_0 = \min\left\{ \frac{2}{\widetilde{C}_{E^0}}, \frac{1}{4E(\Psi^0)} \right\}$ obviously depends on $E(\Psi_0)$ and in fact  $\widetilde{C}_{E^0}$ depends on $k_m$. This completes our proof. 
\end{proof}

\begin{remark}
    It should be noted that in our proof process, the use of the coercivity of $\mathcal{H}_{\Psi^n}$ and $\mathcal{A}_{\Psi^n}^{\tau}(\cdot,\cdot)$ was inevitable. Hence the time step size should satisfy $0< \tau \leq a_{\Psi^n}:=1/C^{\frac{4}{4-d}}_{\Psi^n}$. Noticing $C_{\Psi^n}=Ck_m\|\Psi^n\|^{\frac{d}{2}}_{H^1_0}$ combined with (\ref{2nd-solu H-norm}) implies that $a_{\Psi^n}$ is bounded by $C_{E^n}$, which is also bounded by the expression related to initial energy.
\end{remark}

\section{Global and local convergence  of the GFSI}

The GFSI (\ref{original GFSI}) of the two-component Gross-Pitaevskii energy functional exhibits global convergence, and we further establish a sharp local convergence behavior. We provide the rigorous proof in this section.
\begin{theorem}[{\bf Global convergence}]\label{global convergence}
   Let $\{\Psi^n\}_{n\in\mathbb{N}}$ be the iteration sequence generated by (\ref{original GFSI}) and time step $\tau$ satisfies that $0< \tau \le \tau_0$, $\tau_0$ is given by \eqref{tau0}. Then there exists a subsequence $\{\Psi^{n_j}\}_{j\in\mathbb{N}}$ and some $\Psi_s \in [H^1_0(\Omega)]^2$ such that $\Psi^{n_j}$ converges strongly to $\Psi_s$, which is a stationary state of the coupled Gross-Pitaevskii energy functional, i.e.,
   \begin{equation*}
   	\mathcal{H}_{\Psi_s} \Psi_s = \lambda_s \Psi_s \quad \text{with} \quad \lambda_s = \langle \mathcal{H}_{\Psi_s} \Psi_s, \Psi_s \rangle.
   \end{equation*}
\end{theorem}
To prove this theorem, we need one lemma as follows.
\begin{lemma}\label{LMA: gc pre}
    For any $\Psi \in \mathcal{M}$, $0 < a \leq a_\Psi$, $r \in \{2, 4\}$ with $r' = \frac{r}{r-1}$, and any $\mathbf{u} \in [L^{r'}(\Omega)]^2$, it holds that
\[
    \|(\mathcal{I}+a\mathcal{H}_{\Psi})^{-1}\mathbf{u}\|_{H^1_0} \leq \frac{C}{a}\|\mathbf{u}\|_{r'},
\]
where $C$ is independent of $V$ and $\Psi$.
\end{lemma}
\begin{proof}
    By \eqref{bilinear form}, \eqref{coercivity of bf}, H\"{o}lder's inequality and Sobolev embedding, we have:
    \[\begin{aligned}
        a\|(\mathcal{I}+a\mathcal{H}_{\Psi})^{-1}\mathbf{u}\|_{H^1_0}^2 
        &\leq C\mathcal{A}_{\Psi}^a\left((\mathcal{I}+a\mathcal{H}_{\Psi})^{-1}\mathbf{u},(\mathcal{I}+a\mathcal{H}_{\Psi})^{-1}\mathbf{u}\right)\\
        &= C\left(\mathbf{u},(\mathcal{I}+a\mathcal{H}_{\Psi})^{-1}\mathbf{u}\right)_{L^2}\\
        & \leq C\|(\mathcal{I}+a\mathcal{H}_{\Psi})^{-1}\mathbf{u}\|_r \|\mathbf{u}\|_{r^{\prime}}\\
        & \leq C\|(\mathcal{I}+a\mathcal{H}_{\Psi})^{-1}\mathbf{u}\|_{H^1_0} \|\mathbf{u}\|_{r^{\prime}},
    \end{aligned}
    \]
    in which $r=2 \hspace{0.5em}\text{or}\hspace{0.5em}4$, $r^{\prime}=\frac{r}{r-1}$ and the fact that $(\mathcal{I}+a\mathcal{H}_{\Psi})^{-1}\mathbf{u} \in [H^1_0(\Omega)]^2$ used.
\end{proof}
With this preparatory work, we can prove the Theorem \ref{global convergence} as follows.
\begin{proof}[{\bf Proof of Theorem \ref{global convergence}}]
    Since \eqref{bound of operator}, \eqref{operator boundede below}, and $\|\Psi^n\|_{L^2}=1$ hold, the iteration sequence $\{\Psi^n\}_{n\in\mathbb{N}}$ is bounded in $[H^1_0(\Omega)]^2$. Then, there exists a $\Psi_s \in [H^1_0(\Omega)]^2$ and a subsequence $\{\Psi^{n_j}\}_{j\in\mathbb{N}}$ such that 
    \begin{equation}\label{pre_1}
        \psi_i^{n_j} \rightharpoonup (\Psi_s)_i \quad i=1,2 \quad \text{weakly in} \quad H_0^1(\Omega).
    \end{equation}
The compact embedding $H_0^1(\Omega) \hookrightarrow L^2(\Omega)$ and $H_0^1(\Omega) \hookrightarrow L^4(\Omega)$ then implies that:
\begin{equation}\label{pre_2}
    \psi_i^{n_j} \to (\Psi_s)_i \quad i=1,2 \quad \text{strongly in} \quad L^2(\Omega), \quad L^4(\Omega).
\end{equation}
Specifically, $\|\Psi_s\|_{L^2}=\lim_{j\to \infty}\|\Psi^{n_j}\|_{L^2}=1$, which indicates that $\Psi_s \in\mathcal{M}$.

By the notation in Remark \ref{L-M lemma} and denoting $\widetilde{\Psi}_s=(\mathcal{I}+\tau \mathcal{H}_{\Psi_s})^{-1}\Psi_s \in [H^1_0(\Omega)]^2$, we have
\[\begin{aligned}
    \widetilde{\Psi}^{n_j+1}-\widetilde{\Psi}_s
    & =(\mathcal{I}+\tau\mathcal{H}_{\Psi^{n_j}})^{-1}\Psi^{n_j}- (\mathcal{I}+\tau \mathcal{H}_{\Psi_s})^{-1}\Psi_s \\
    & =\underbrace{(\mathcal{I}+\tau\mathcal{H}_{\Psi^{n_j}})^{-1}(\Psi^{n_j}-\Psi_s)}_{I_1} + \underbrace{\left[(\mathcal{I}+\tau\mathcal{H}_{\Psi^{n_j}})^{-1}-(\mathcal{I}+\tau \mathcal{H}_{\Psi_s})^{-1} \right]\Psi_s}_{I_2}.
\end{aligned}\]
According to Lemma \ref{LMA: gc pre}, when $r^{\prime}=2$ we know that
\begin{equation}\label{pre_3}
    \|I_1\|_{H^1_0} \leq C\|\Psi^{n_j}-\Psi_s\|_{L^2}.
\end{equation}
And by direct calculation and (\ref{additive structure}), we get
\[\begin{aligned}
    I_2 
    &= \left[(\mathcal{I}+\tau\mathcal{H}_{\Psi^{n_j}})^{-1}-(\mathcal{I}+\tau \mathcal{H}_{\Psi_s})^{-1} \right]\Psi_s \\
    &= \left[(\mathcal{I}+\tau\mathcal{H}_{\Psi^{n_j}})^{-1} (\mathcal{I}+\tau \mathcal{H}_{\Psi_s})(\mathcal{I}+\tau \mathcal{H}_{\Psi_s})^{-1} - (\mathcal{I}+\tau \mathcal{H}_{\Psi_s})^{-1} \right]\Psi_s \\
    &= \left[ (\mathcal{I}+\tau\mathcal{H}_{\Psi^{n_j}})^{-1}(\mathcal{I}+\tau \mathcal{H}_{\Psi_s} - \mathcal{I}+\tau\mathcal{H}_{\Psi^{n_j}})(\mathcal{I}+\tau \mathcal{H}_{\Psi_s})^{-1} \right]\Psi_s \\
    &= \tau (\mathcal{I}+\tau\mathcal{H}_{\Psi^{n_j}})^{-1} \left( \left( \rho_1(\Psi_s)-\rho_1(\Psi^{n_j}) \right)(\widetilde{\Psi}_s)_1,\left( \rho_2(\Psi_s)-\rho_2(\Psi^{n_j}) \right)(\widetilde{\Psi}_s)_2\right)^{\top}.
\end{aligned}
\]
By (\ref{density}) and H\"{o}lder's inequality, noticing 
\begin{equation}\label{pre_4}
\begin{aligned}
     \|\rho_i(\Psi_s)- &\rho_i(\Psi^{n_j})\|_2
      \leq \left\|\sum_{l=1,2}|k_{il}| \left( (\Psi_s)^2_l-(\Psi^{n_j})^2_l\right) \right\|_{L^2} \\
     & \leq \sum_{l=1,2} k_m \left\|\Psi_s+\Psi^{n_j}\right\|_4\cdot \left\|\Psi_s-\Psi^{n_j}\right\|_4\\
     & \leq C \left( \sum_{l=1,2}\|\Psi_s+\Psi^{n_j}\|_4\right) \left( \sum_{l=1,2}\|\Psi_s-\Psi^{n_j} \|_4 \right) \\
     & \leq C \left\|\Psi_s + \Psi^{n_j}\right\|_4 \left\|\Psi_s - \Psi^{n_j}\right\|_4.
\end{aligned}
\end{equation}
Then, according to (\ref{pre_4}), Lemma \ref{LMA: gc pre}, H\"{o}lder's inequality and Sobolev embedding, we have
\begin{equation}\begin{aligned}\label{pre_5}
    \|I_2\|_{H^1_0} & \leq C \sum_{i=1,2} \left\| \rho_i(\Psi_s)-\rho_i(\Psi^{n_j}) \right\|_2 \|(\widetilde{\Psi}_s)_i \|_4 \\
    & \leq C \|\Psi_s + \Psi^{n_j}\|_4 \|\Psi_s - \Psi^{n_j}\|_4 \sum_{i=1,2} \|(\widetilde{\Psi}_s)_i \|_4 \\
    & \leq C \|\Psi_s + \Psi^{n_j}\|_4 \|\Psi_s - \Psi^{n_j}\|_4 \|\widetilde{\Psi}_s \|_4 \leq C  \|\Psi_s - \Psi^{n_j}\|_4.
\end{aligned}
\end{equation}
Combining (\ref{pre_2}), (\ref{pre_3}) and (\ref{pre_5}) and the fact that $\{\Psi^n \}_n$ is bounded in $[H^1_0(\Omega)]^2$, we can immediately prove the fact that
\[
    \widetilde{\Psi}^{n_j+1} \to \widetilde{\Psi}_s, \quad 
\text{strongly in} \quad [H^1_0(\Omega)]^2,\quad j \to \infty
\]
which further indicates
\[
    \Psi^{n_j+1}= \frac{\widetilde{\Psi}^{n_j+1}}{\|\widetilde{\Psi}^{n_j+1}\|_{L^2}} \to \frac{\widetilde{\Psi}_s}{\|\widetilde{\Psi}_s\|_{L^2}}, \quad 
\text{strongly in} \quad [H^1_0(\Omega)]^2.
\]

Noting (\ref{energy dis}) further implies that
\[
    \lim_{n \to\infty}\|\Psi^{n+1} - \Psi^n\|_{H^1_0} \leq \lim_{n \to \infty} C\sqrt{ E(\Psi^n)-E(\Psi^{n+1})} = 0,
\]
hence $\Psi^{n_j}$ has the same strong limit as $\Psi^{n_j+1}$, i.e. $\Psi^{n_j} \to \frac{\widetilde{\Psi}_s}{\|\widetilde{\Psi}_s\|_{L^2}}$, strongly in $[H^1_0(\Omega)]^2$. Then combined with (\ref{pre_5}) we have $\Psi_s=\frac{\widetilde{\Psi}_s}{\|\widetilde{\Psi}_s\|_{L^2}}$. Recalling the definition of $\widetilde{\Psi}_s$ previously stated implies 
\[
    \|\widetilde{\Psi}_s\|_{L^2} \Psi_s = (\mathcal{I}+\tau \mathcal{H}_{\Psi_s})^{-1} \Psi_s,
\]
acting with $\mathcal{I}+\tau \mathcal{H}_{\Psi_s}$ on the both sides, we obtain $\mathcal{H}_{\Psi_s}\Psi_s = \frac{1-\|\widetilde{\Psi}_s\|_{L^2}}{\tau \|\widetilde{\Psi}_s\|_{L^2}} \Psi_s$, then $\lambda_s = \frac{1-\|\widetilde{\Psi}_s\|_{L^2}}{\tau \|\widetilde{\Psi}_s\|_{L^2}} = \langle \mathcal{H}_{\Psi_s}\Psi_s,\Psi_s \rangle$.
\end{proof}

Finally, we establish the connection between GFSI and preconditioned Riemannian gradient descent (P-RG). Specifically, GFSI is locally equivalent to P-RG with the preconditioner chosen as $\mathcal{P}_{\Psi} = \mathcal{A}_{\Psi}^{\tau}$, whose update rule reads
\begin{align}\label{eq:P-RG}
    \Psi^{n+1} &= \frac{\Psi^{n} - \tau \, {\rm Proj}_{\Psi^n}^{\mathcal{A}_{\Psi^n}^{\tau}} \bigl( (\mathcal{A}_{\Psi^{n}}^{\tau})^{-1} \mathcal{H}_{\Psi^{n}}\Psi^{n} \bigr)}{\bigl\| \Psi^{n} - \tau \, {\rm Proj}_{\Psi^n}^{\mathcal{A}_{\Psi^n}^{\tau}} \bigl( (\mathcal{A}_{\Psi^{n}}^{\tau})^{-1} \mathcal{H}_{\Psi^{n}}\Psi^{n} \bigr) \bigr\|_{L^2}},
\end{align}
where the projection operator is defined by
\begin{align}\label{eq:projection}
    {\rm Proj}_{\Psi}^{\mathcal{A}_{\Psi}^{\tau}} \mathbf{v} = \mathbf{v} - \frac{( \mathbf{v}, \Psi )_{L^2}}{( (\mathcal{A}_{\Psi}^{\tau})^{-1}\mathcal{I}\Psi, \Psi )_{L^2}} (\mathcal{A}_{\Psi}^{\tau})^{-1}\mathcal{I}\Psi.
\end{align}
And the projection operator ${\rm Proj}_{\Psi}^{\mathcal{A}_{\Psi}^{\tau}}$ is essentially an operator defined as $ [H_0^1(\Omega)]^2 \to T_{\Psi}\mathcal{M}$, where $T_{\Psi}\mathcal{M}$ represents the tangent space corresponding to manifold $\mathcal{M}$ and can be given by
\begin{equation}
    T_{\Psi}\mathcal{M} := \left\{ \mathbf{v} \in [H_0^1(\Omega)]^2 \mid (\Psi, \mathbf{v})_{L^2} = 0 \right\}.
\end{equation}

The following theorem shows that GFSI and this P-RG formulation share the same linearized dynamics at the stationary state, and hence exhibit identical local convergence behavior.

\begin{theorem}\label{thm:local-equivalence}
    Let $\{\Psi^n\}_{n\in\mathbb{N}}$ be a sequence generated by GFSI, $\Psi_s$ a stationary state, and $\Phi$ the one-step update map of GFSI. The following statements hold:
    \begin{enumerate}
        \item[(i)] \textcolor{blue}{The $H_0^1$-norm between consecutive iterates controls the Riemannian gradient in the P-RG algorithm \eqref{eq:P-RG}, i.e.,}
        \[ \color{blue} \left\|{\rm Proj}_{\Psi^n}^{\mathcal{A}_{\Psi^n}^{\tau}} \bigl( (\mathcal{A}_{\Psi^{n}}^{\tau})^{-1} \mathcal{H}_{\Psi^{n}}\Psi^{n} \bigr)\right\|_{\mathcal{A}^{\tau}_{\Psi^n}}
        \le C\|\Psi^{n+1}-\Psi^n\|_{H_0^1};
        \]
        \item[(ii)] The Fr\'echet derivative of $\Phi$ at $\Psi_s$ along the tangent space $T_{\Psi_s}\mathcal{M}$ satisfies
        \[
            \Phi'(\Psi_s) = I-\tau{\rm Proj}_{\Psi_s}^{\mathcal{A}_{\Psi_s}^{\tau}}(\mathcal{A}_{\Psi_s}^{\tau})^{-1}\bigl(\tfrac{1}{2}E''(\Psi_s)-\lambda_s\mathcal{I}\bigr).
        \]
    \end{enumerate}
\end{theorem}

By the above theorem, GFSI and P-RG share the same local convergence behavior in a neighborhood of $\Psi_s$.
\begin{proof}
\begingroup
\color{blue}
For $(i)$, take the $L^{2}$ inner product of \eqref{GFSI} with any $\mathbf{v}\in T_{\Psi^n}\mathcal{M}$, we have: 
\begin{equation}\label{eq:SC-1_new}
\frac{1}{\tau}(\Psi^{n+1}-\Psi^n,\mathbf{v})_{L^2}
=-\langle\mathcal{H}_{\Psi^n}\Psi^{n+1},\mathbf{v}\rangle+\lambda^{n+1}(\Psi^n,\mathbf{v})_{L^2}
=-\langle\mathcal{H}_{\Psi^n}\Psi^{n+1},\mathbf{v}\rangle.
\end{equation}
Writing $\Psi^{n+1}=\Psi^n+(\Psi^{n+1}-\Psi^n)$ in \eqref{eq:SC-1_new} gives
\[
((\mathcal{A}_{\Psi^n}^{\tau})^{-1}\mathcal{H}_{\Psi^n}\Psi^n,\mathbf{v})_{\mathcal{A}_{\Psi^n}^{\tau}}=
-\frac{1}{\tau}(\Psi^{n+1}-\Psi^n,\mathbf{v})_{L^2}
-\langle\mathcal{H}_{\Psi^n}(\Psi^{n+1}-\Psi^n),\mathbf{v}\rangle.
\]
Since $\text{Proj}^{\mathcal{A}^{\tau}_{\Psi^n}}_{\Psi^n}$ is the $\mathcal{A}_{\Psi^n}^{\tau}$-orthogonal projection onto
$T_{\Psi^n}\mathcal{M}$, we have
\begin{equation}\label{eq:SC-2_new}
\left({\rm Proj}_{\Psi^n}^{\mathcal{A}_{\Psi^n}^{\tau}} \bigl( (\mathcal{A}_{\Psi^{n}}^{\tau})^{-1} \mathcal{H}_{\Psi^{n}}\Psi^{n} \bigr),\mathbf{v}\right)_{\mathcal{A}_{\Psi^n}^{\tau}}
=
-\frac{1}{\tau}(\Psi^{n+1}-\Psi^n,\mathbf{v})_{L^2}
-\langle\mathcal{H}_{\Psi^n}(\Psi^{n+1}-\Psi^n),\mathbf{v}\rangle.
\end{equation}
For all $\mathbf{v}\in T_{\Psi^n}\mathcal{M}$ and fixed $\tau$, we have:
\[
\frac{1}{\tau}|(\Psi^{n+1}-\Psi^n,\mathbf{v})_{L^2}|
\le
C\,\|\Psi^{n+1}-\Psi^n\|_{H_0^1}\,\|\mathbf{v}\|_{H_0^1}.
\]
On the other hand, by the continuity of $\mathcal{H}_{\Psi^n}$,
\[
|\langle\mathcal{H}_{\Psi^n}(\Psi^{n+1}-\Psi^n),\mathbf{v}\rangle|
\le
C_{\Psi^n}\,\|\Psi^{n+1}-\Psi^n\|_{H_0^1}\,\|\mathbf{v}\|_{H_0^1}.
\]
Since $\{E(\Psi^n)\}$ is bounded, we obtain the uniform boundedness of $\{\Psi^n\}$ in $[H_0^1(\Omega)]^2$ by Lemma \ref{pre-lemma 2}; there exists $C>0$ such that $C_{\Psi^n}\le C$ for all $n$. Combining these estimates with \eqref{eq:SC-2_new}, then for all
$\mathbf{v}\in T_{\Psi^n}\mathcal{M}$,
\[
\left|\left({\rm Proj}_{\Psi^n}^{\mathcal{A}_{\Psi^n}^{\tau}} \bigl( (\mathcal{A}_{\Psi^{n}}^{\tau})^{-1} \mathcal{H}_{\Psi^{n}}\Psi^{n} \bigr),\mathbf{v}\right)_{\mathcal{A}_{\Psi^n}^{\tau}}\right|
\le
C\,
\|\Psi^{n+1}-\Psi^n\|_{H_0^1}\,\|\mathbf{v}\|_{H_0^1}.
\]
Considering $\mathbf{v} = {\rm Proj}_{\Psi^n}^{\mathcal{A}_{\Psi^n}^{\tau}} \bigl((\mathcal{A}_{\Psi^{n}}^{\tau})^{-1} \mathcal{H}_{\Psi^{n}}\Psi^{n} \bigr)$ and the fact that $\|\cdot\|_{H^1}$ and $\|\cdot\|_{\mathcal{A}_{\Psi^n}^{\tau}}$ are equivalent norms yields
\[
\left\|{\rm Proj}_{\Psi^n}^{\mathcal{A}_{\Psi^n}^{\tau}} \bigl( (\mathcal{A}_{\Psi^{n}}^{\tau})^{-1} \mathcal{H}_{\Psi^{n}}\Psi^{n} \bigr)\right\|_{\mathcal{A}_{\Psi^n}^{\tau}}
\le
C\,\|\Psi^{n+1}-\Psi^n\|_{H_0^1}.
\]
\endgroup
For $(ii)$, define the nonlinear part of $\mathcal{H}_{\Psi}$ as the operator $\mathcal{F}(\Psi) \colon [H_0^1(\Omega)]^2 \to H^{*}$, given by
\begin{align*}
    \mathcal{F}(\Psi) = 
    \begin{pmatrix}
        k_{11}|\psi_1|^2 + k_{12}|\psi_2|^2 \\
        k_{12}|\psi_1|^2 + k_{22}|\psi_2|^2
    \end{pmatrix} \odot\, .
\end{align*}
Its Fr\'echet derivative with respect to $\Psi$ in the direction $\mathbf{v} = (v_1, v_2)^\top \in T_{\Psi}\mathcal{M}$ is computed as
\begin{align*}
    \mathcal{F}'(\Psi)\mathbf{v} = 
    \begin{pmatrix}
        k_{11}\psi_1\overline{{\cdot}} + k_{11}\overline{\psi}_1 & k_{12}\psi_2\overline{{\cdot}} + k_{12}\overline{\psi}_2 \\
        k_{12}\psi_1\overline{{\cdot}} + k_{12}\overline{\psi}_1 & k_{22}\psi_2\overline{{\cdot}} + k_{22}\overline{\psi}_2
    \end{pmatrix} \mathbf{v} \odot .
\end{align*}
Differentiating the identity $\mathcal{A}_{\Psi}^{\tau}\bigl((\mathcal{A}_{\Psi}^{\tau})^{-1}\Psi\bigr) = \Psi$ with respect to $\Psi$ in the direction $\mathbf{v}$ yields
\begin{align*}
    \bigl((\mathcal{A}_{\Psi}^{\tau})^{-1}\Psi\bigr)'\mathbf{v}
    &= (\mathcal{A}_{\Psi}^{\tau})^{-1}\Bigl(\mathbf{v} - \tau\,\mathcal{F}'(\Psi)\mathbf{v} \odot (\mathcal{A}_{\Psi}^{\tau})^{-1}\Psi\Bigr) \\
    &= (\mathcal{A}_{\Psi}^{\tau})^{-1}\Bigl(\mathcal{I} - \tau\bigl((\mathcal{A}_{\Psi}^{\tau})^{-1}\Psi\bigr) \odot \mathcal{F}'(\Psi)\Bigr)\mathbf{v},
\end{align*}
where the second line follows from the commutativity of the Hadamard product and the linearity of $\mathcal{F}'(\Psi)$ in $\mathbf{v}$.

Since $\Phi(\Psi) = (\mathcal{A}_{\Psi}^{\tau})^{-1}\Psi \big/ \|(\mathcal{A}_{\Psi}^{\tau})^{-1}\Psi\|_{L^2}$, applying the quotient rule gives
  \begin{align*}
    \Phi'(\Psi)\mathbf{v}
    &= \frac{\bigl((\mathcal{A}_{\Psi}^{\tau})^{-1}\Psi\bigr)'\mathbf{v}}{\|(\mathcal{A}_{\Psi}^{\tau})^{-1}\Psi\|_{L^2}} 
       - \frac{\bigl((\mathcal{A}_{\Psi}^{\tau})^{-1}\Psi,\; \bigl((\mathcal{A}_{\Psi}^{\tau})^{-1}\Psi\bigr)'\mathbf{v}\bigr)_{L^2}}{\|(\mathcal{A}_{\Psi}^{\tau})^{-1}\Psi\|_{L^2}^3} \, (\mathcal{A}_{\Psi}^{\tau})^{-1}\Psi \\[4pt]
    &= \frac{1}{\|(\mathcal{A}_{\Psi}^{\tau})^{-1}\Psi\|_{L^2}} \, \mathrm{Proj}_{\Phi(\Psi)}^{L^2}\Bigl(\bigl((\mathcal{A}_{\Psi}^{\tau})^{-1}\Psi\bigr)'\mathbf{v}\Bigr) \\[4pt]
    &= \frac{1}{\|(\mathcal{A}_{\Psi}^{\tau})^{-1}\Psi\|_{L^2}} \, \mathrm{Proj}_{\Phi(\Psi)}^{L^2}\biggl((\mathcal{A}_{\Psi}^{\tau})^{-1}\Bigl(\mathcal{I} - \tau\bigl((\mathcal{A}_{\Psi}^{\tau})^{-1}\Psi\bigr) \odot \mathcal{F}'(\Psi)\Bigr)\mathbf{v}\biggr).
\end{align*}
At the stationary state $\Psi_s$, the eigenvalue equation $\mathcal{H}_{\Psi_s}\Psi_s = \lambda_s\Psi_s$ implies
\begin{align}\label{Proj}
    (\mathcal{A}_{\Psi_s}^{\tau})^{-1}\Psi_s = \frac{1}{1+\tau\lambda_s}\Psi_s, \qquad \Phi(\Psi_s) = \Psi_s.
\end{align}
Substituting into the bracketed operator, we obtain
\begin{align*}
    \mathcal{I} - \tau\bigl((\mathcal{A}_{\Psi_s}^{\tau})^{-1}\Psi_s\bigr) \odot \mathcal{F}'(\Psi_s)
    &= \mathcal{I} - \frac{\tau}{1+\tau\lambda_s}\,\Psi_s \odot \mathcal{F}'(\Psi_s).
\end{align*}
Noting that $\Psi \odot \mathcal{F}'(\Psi) = \tfrac{1}{2}E''(\Psi) - \mathcal{H}_{\Psi}$ by
\[
    (\mathcal{F}(\Psi)\odot\Psi)'  = \Psi\odot \mathcal{F}(\Psi)' + \mathcal{F}(\Psi)\odot,
\]
 we continue:
\begin{align*}
    \mathcal{I} - \tau\bigl((\mathcal{A}_{\Psi_s}^{\tau})^{-1}\Psi_s\bigr) \odot \mathcal{F}'(\Psi_s)&= \mathcal{I} - \frac{\tau}{1+\tau\lambda_s}\bigl(\tfrac{1}{2}E''(\Psi_s) - \mathcal{H}_{\Psi_s}\bigr) \\[4pt]
    &= \frac{1}{1+\tau\lambda_s}\Bigl((1+\tau\lambda_s)\mathcal{I} - \tau\bigl(\tfrac{1}{2}E''(\Psi_s) - \mathcal{H}_{\Psi_s}\bigr)\Bigr) \\[4pt]
    &= \frac{1}{1+\tau\lambda_s}\Bigl(\mathcal{I} + \tau\mathcal{H}_{\Psi_s}  - \tau\bigl(\tfrac{1}{2}E''(\Psi_s) - \lambda_s\mathcal{I}\bigr)\Bigr).
\end{align*}
Furthermore, by \eqref{Proj}, we have $\text{Proj}_{\Psi_s}^{L^2}=\text{Proj}_{\Psi_s}^{\mathcal{A}_{\Psi_s}^{\tau}}$. Collecting all terms, the Fr\'echet derivative at $\Psi_s$ simplifies to
\[
    \Phi'(\Psi_s) = I - \tau \, \mathrm{Proj}_{\Psi_s}^{\mathcal{A}_{\Psi_s}^{\tau}} \, (\mathcal{A}_{\Psi_s}^{\tau})^{-1} \bigl(\tfrac{1}{2}E''(\Psi_s) - \lambda_s \mathcal{I}\bigr),
\]
which completes the proof.
\end{proof}

\begin{lemma}\label{LGI}
Assume that $\Psi_s$ is a constrained stationary state of the energy functional $E$, i.e.,
\[
\mathcal{H}_{\Psi_s}\Psi_s=\lambda_s\Psi_s\quad \text{with}\quad\lambda_s=\langle\mathcal{H}_{\Psi_s}\Psi_s,\Psi_s\rangle.
\]
Then there exist constants $C>0$, $\sigma>0$, and $\theta\in\bigl(0,\tfrac12\bigr]$ such that for all $\Psi\in\mathcal M$ with
$\|\Psi-\Psi_s\|_{H_0^1}\le\sigma$, 
\[
|E(\Psi)-E(\Psi_s)|^{1-\theta}
\le
C\,\left\|{\rm Proj}^{\mathcal{A}^{\tau}_{\Psi}}_{\Psi}\bigl((\mathcal{A}_{\Psi}^{\tau})^{-1} \mathcal{H}_{\Psi}\Psi \bigr)\right\|_{\mathcal{A}^{\tau}_{\Psi}}.
\]
\end{lemma}
\begin{proof}
Since both the nonlinear terms and the constraint manifold of the model are real analytic, the proof follows from the abstract \L{}ojasiewicz--Simon gradient inequality on Hilbert manifolds (see~\cite[Corollary~5.2 and Remark~5.4]{Rupp2020} for details).
\end{proof}

Moreover, we emphasize that when $\theta = 1/2$ and $\Psi_s$ is a ground state, this inequality is equivalent to the local Morse--Bott condition Gross--Pitaevskii energy functional (a rigorous proof for the single-component case can be found in~\cite[Lemmas~3.2--3.4]{2026Structure}; \textcolor{blue}{the same derivation applies here as only scalar/vector differs, with identical functional regularity and constraint manifold structure}). 

\begin{definition}
\textcolor{blue}{We say that the energy functional $E$ satisfies the local Morse--Bott condition at a ground state $\Psi_g$ if there exists $\sigma > 0$ such that the set of ground states
  \begin{align}\label{Local-Manifold}
    \mathcal{S}_{\sigma}(\Psi_g) := \bigl\{ \Psi \in \mathcal{M} : \Psi \text{ is a ground state and } \|\Psi - \Psi_g\|_{H_0^1} < \sigma \bigr\}
  \end{align}
  is a finite-dimensional $C^1$ embedded submanifold of $\mathcal{M}$, and for every $\Psi \in \mathcal{S}_{\sigma}(\Psi_g)$, the kernel of the Riemannian Hessian coincides with the tangent space of $\mathcal{S}_{\sigma}(\Psi_g)$, i.e.,
  \begin{align}\label{Ker-Hess}
    \ker\Bigl( \bigl(E''(\Psi) - \lambda_{\Psi} \mathcal{I}\bigr)\big|_{T_\Psi \mathcal{M}} \Bigr) = T_\Psi \mathcal{S}_{\sigma}(\Psi_g).
  \end{align}}

\end{definition}

\begin{corollary}\label{coroll1}
Let $\{\Psi^n\}_{n\in \mathbb{N}}$ be the sequence generated by \eqref{GFSI} and time step $\tau$ satisfies that $0< \tau \le \tau_0$, $\tau_0$ is given by \eqref{tau0}. There exists a stationary state $\Psi_s \in \mathcal{M}$ of the Gross--Pitaevskii energy functional such that $\Psi_s$ is the unique limit point of $\{\Psi^n\}$ in the $\|\cdot\|_{H_0^1}$-norm. Moreover, $\{\Psi^n\}$ converges to $\Psi_s$ with the following asymptotic rate:
\begin{equation}\label{General-rate}
    \|\Psi^n - \Psi_s\|_{H_0^1} \le
    \begin{cases}
        C \rho^n, & \text{if } \theta = \frac{1}{2}, \\[6pt]
        C n^{-\frac{\theta}{1-2\theta}}, & \text{if } \theta \in \bigl(0, \frac{1}{2}\bigr),
    \end{cases}
\end{equation}
where $\rho \in (0,1)$ depends on the step size $\tau$. 
\end{corollary} 
\begin{proof}
The global convergence of the full sequence and the coarse local convergence rate \eqref{General-rate} follow from standard arguments based on the {\L}ojasiewicz--Simon inequality together with Theorems~\ref{THM} and \ref{thm:local-equivalence}-$(i)$. \textcolor{blue}{These results verify the three conditions required in \cite[Theorem~2.1]{Zhang2022}, to which we refer for further details.}
\end{proof}

\begin{corollary}\label{coroll2}
 Suppose that $\Psi_g$ is a ground state and that the energy functional $E$ satisfies the Morse--Bott condition at $\Psi_g$ (i.e., the {\L}ojasiewicz--Simon inequality holds with exponent $\theta = 1/2$), this convergence rate can be sharpened to an optimal asymptotic bound. Specifically, let $\rho_0$ denote the spectral radius of the linearized operator restricted to the subspace $R_{\Psi_g}$:
\begin{align}\label{rho-less-1}
    \rho_0 := \rho\!\left( \left. \bigl(I - \tau (\mathcal{A}_{\Psi_g}^{\tau})^{-1}
    (\tfrac{1}{2}E''(\Psi_g) - \lambda_g \mathcal{I})\bigr) \right|_{R_{\Psi_g}} \right),
\end{align}
where
\[
    R_{\Psi_g} := \left\{ \mathbf{v} \in T_{\Psi_g}\mathcal{M} \;\middle|\; \mathcal{A}_{\Psi_g}^{\tau}(\mathbf{v}, \mathbf{u}) = 0,\; \forall\, \mathbf{u} \in T_{\Psi_g} \mathcal{S}_{\sigma}(\Psi_g) \right\}.
\]
Then, for any sufficiently small $\varepsilon > 0$, we have the sharp estimate
\begin{equation}
    \rho = \rho_0 + \varepsilon.
\end{equation}
In particular, when the initial data lies in a sufficiently small neighborhood of the stationary state $\Psi_g$, the admissible range of $\tau$ admits the explicit characterization
\begin{equation}\label{tau-expli}
 \tau\in
    \begin{cases}
        (0,\infty), & \text{if } \xi\in(0,2], \\[6pt]
        (0,\tau_0), & \text{if } \xi\in(2,\infty),
    \end{cases}
\end{equation}
where
\[
    \xi := \sup_{\mathbf{v}\in T_{\Psi_g}\mathcal{M}}  \frac{\langle (\tfrac{1}{2}E''(\Psi_g)-\lambda_g\mathcal{I})\mathbf{v},\mathbf{v}\rangle}{\langle \mathcal{H}_{\Psi_g}\mathbf{v},\mathbf{v}\rangle}.
\]
\end{corollary}
\begin{proof}
     Regarding the sharp convergence rate under the Morse--Bott condition, the proof follows verbatim from \cite[Theorem~3.2]{2026Structure}; we therefore omit the repetitive details, \textcolor{blue}{as the derivation carries over directly with identical functional regularity and constraint manifold structure, differing only in the scalar versus vector setting}. 

\textcolor{blue}{We now derive an explicit characterization of the admissible step size range in a neighborhood of $\Psi_g$. By Theorem~\ref{thm:local-equivalence}-$(ii)$, for any $\Psi \in \mathcal{S}_{\sigma}(\Psi_g)$, we have
\begin{align}\label{GFSI-P-RG}
    \nonumber\Psi^{n+1} - &\Psi = \Phi(\Psi^n)-\Phi(\Psi)\\
    \nonumber&= \Psi^n - \Psi - \tau {\rm Proj}^{\mathcal{A}^{\tau}_{\Psi}}_{\Psi}(\mathcal{A}_{\Psi}^{\tau})^{-1} \bigl(\tfrac{1}{2}E''(\Psi) -\lambda_{\Psi}\mathcal{I}\bigr)(\Psi^n-\Psi) + o(\tau\|\Psi^n-\Psi\|_{H_0^1})\\
    &= \Psi^n - \Psi - \tau {\rm Proj}^{\mathcal{A}^{\tau}_{\Psi^n}}_{\Psi^n}\bigl((\mathcal{A}_{\Psi^n}^{\tau})^{-1} \mathcal{H}_{\Psi^n}\Psi^n \bigr) + o(\tau\|\Psi^n-\Psi\|_{H_0^1}).
\end{align}
Moreover, \cite[Lemma~3.2]{2026Structure} applies directly to the present setting for the same reason. According to Eq.~(3.14) in \cite{2026Structure}, for any $\Psi^n$ in a sufficiently small neighborhood of $\Psi_g$, there exists $\Psi_g' \in \mathcal{S}_{\sigma}(\Psi_g)$ such that
\begin{align*}
    \Psi^n-\Psi_g' = \mathcal{J}_{\Psi_g'}(\Psi^n-\Psi_g') + o(\|\Psi^n-\Psi_g'\|_{H_0^1}),\; \|\Psi^n-\Psi_g'\|_{H_0^1} \le C\|\Psi^n-\Psi_g\|_{H_0^1},
\end{align*}
where $\mathcal{J}_{\Psi_g'}$ denotes the orthogonal projection onto $R_{\Psi_g'}$. Combining the coercivity of the Riemannian Hessian on $R_{\Psi'_g}$, which follows from the finite-dimensional kernel, with a Taylor expansion, we obtain
\begin{align*}
    E(\Psi^n) &-E(\Psi_g)= E(\Psi^n)-E(\Psi_g') \\
    &=\bigl\langle \bigl(\tfrac{1}{2}E''(\Psi_g')-\lambda_{\Psi_g'}\mathcal{I}\bigr)\mathcal{J}_{\Psi_g'}(\Psi^n-\Psi_g'),\,\mathcal{J}_{\Psi_g'}(\Psi^n-\Psi_g') \bigr\rangle + o\bigl(\|\Psi^n-\Psi_g'\|^2_{H_0^1}\bigr) \\
    &\ge C \|\Psi^n-\Psi_g'\|^2_{H_0^1}.
\end{align*}
Furthermore, by the {\L}ojasiewicz--Simon gradient inequality with $\theta = 1/2$, we get
\[
    \|\Psi^n-\Psi_g'\|_{H_0^1} \le C \left\| {\rm Proj}^{\mathcal{A}^{\tau}_{\Psi^n}}_{\Psi^n} \bigl( (\mathcal{A}_{\Psi^n}^{\tau})^{-1} \mathcal{H}_{\Psi^n}\Psi^n \bigr) \right\|_{\mathcal{A}^{\tau}_{\Psi^n}}.
\]
 Applying \eqref{GFSI-P-RG} at $\Psi_g'$, together with the Lipschitz continuity of $\operatorname{Proj}^{\mathcal{A}^{\tau}_{\Psi}}_{\Psi}$ and $E''(\Psi)-\lambda_{\Psi}\mathcal{I}$ with respect to $\Psi$, and expanding the energy $E(\Psi^{n+1})$, for given $\tau>0$, we obtain
\begin{align*}
    E(\Psi^{n+1}) &= E(\Psi^n) - 2\tau \|\mathbf{d}_n\|_{\mathcal{A}^{\tau}_{\Psi_g}}^2 + \tau^2\langle (\tfrac{1}{2}E''(\Psi_g)-\lambda_g\mathcal{I})\mathbf{d}_n,\mathbf{d}_n\rangle + o(\tau^2\|\mathbf{d}_n\|^2_{H_0^1}),
\end{align*}}
\textcolor{blue}{where $\mathbf{d}_n={\rm Proj}^{\mathcal{A}^{\tau}_{\Psi^n}}_{\Psi^n} \bigl( (\mathcal{A}_{\Psi^n}^{\tau})^{-1} \mathcal{H}_{\Psi^n}\Psi^n \bigr)$.
Therefore, the standard exact step size condition $\tau \in (0, 2/L_{\tau})$ (ensuring $2\tau-\tau^2L_{\tau}>0$) for preconditioned Riemannian gradient descent also applies to GFSI}. It remains to estimate $L_{\tau}$, defined by
\[
    L_{\tau} := \sup_{\mathbf{v}\in T_{\Psi_g}\mathcal{M}} \frac{\langle (\tfrac{1}{2}E''(\Psi_g)-\lambda_g\mathcal{I})\mathbf{v},\mathbf{v}\rangle}{\langle \mathcal{A}^{\tau}_{\Psi_g}\mathbf{v},\mathbf{v}\rangle}.
\]
Observe that the restricted operator on the tangent space $T_{\Psi_s}\mathcal{M}$ admits the decomposition
\[
(\mathcal{A}_{\Psi_g}^{\tau})^{-1}\Bigl(\tfrac{1}{2}E''(\Psi_g)-\lambda_g\mathcal{I}\Bigr) =  \frac{1}{\tau}I + (\mathcal{A}_{\Psi_g}^{\tau})^{-1}\Bigl(\Psi_g\odot\mathcal{F}'(\Psi_g)-\Bigl(\lambda_g+\frac{1}{\tau}\Bigr)\mathcal{I}\Bigr).
\]
Since the second term on the right-hand side is a compact operator, the overall operator is a compact perturbation of the scaled identity $\frac{1}{\tau}I$. Consequently, its spectrum consists solely of isolated eigenvalues of finite multiplicity and a unique accumulation point at $\frac{1}{\tau}$, and $L_{\tau}$ is precisely the eigenvalue of largest modulus.

\textcolor{blue}{Let $\mathbf{v}_{\max}$ be the $L^2$-normalized eigenfunction corresponding to the eigenvalue of largest, and define
$\nu_{\mathbf{v}_{\max}} := \langle \mathcal{H}_{\Psi_g} \mathbf{v}_{\max}, \mathbf{v}_{\max} \rangle.$
Then, we have} 
\begin{align*}
    L_{\tau}  =  \frac{\langle (\tfrac{1}{2}E''(\Psi_g)-\lambda_g\mathcal{I})\mathbf{v}_{\max},\mathbf{v}_{\max}\rangle}{1 + \tau\nu_{\mathbf{v}_{\max}}}  \le \frac{\xi\nu_{\mathbf{v}_{\max}}}{1 + \tau\nu_{\mathbf{v}_{\max}}}.
\end{align*}
\textcolor{blue}{Substituting this upper bound into the step size condition $\tau \in (0, 2/L_{\tau})$ yields
\[
 \tau < \frac{2(1 + \tau\nu_{\mathbf{v}_{\max}})}{\xi\nu_{\mathbf{v}_{\max}}} \iff (\xi-2)\tau\nu_{\mathbf{v}_{\max}} < 2.
\]
This implies that the condition is satisfied for all $\tau > 0$ whenever $\xi \in (0,2]$. For $\xi > 2$,  Theorems~\ref{THM} and~\ref{thm:local-equivalence}-$(i)$ guarantee the existence of an interval $(0, \tau_0)$ over which the dissipation property of the Riemannian gradient is preserved.}
\end{proof}
For the single-component model, \cite[Lemma~3.5]{ref_arXiv} establishes that $\xi \in (0,2]$.
\begin{remark}
    Finally, we summarize the theoretical context and generalizations of our framework. Regarding model extensions, although this work focuses on two-component BECs with a Josephson junction, our analytical framework directly generalizes to general $p$-component systems ($p>2$) without Josephson junctions, where the mass of each component is individually conserved \cite{ref_multi}. In such regimes, the GFSI scheme naturally decouples into component-wise implicit updates $\frac{\widetilde{\psi}^{n+1}_i - \psi_i^n}{\tau} = -\mathcal{H}_{\psi_i^n}\widetilde{\psi}^{n+1}_i$, followed by individual $L^2$-normalizations. Because the corresponding Hamiltonian $\mathcal{H}_{\psi_i} = -\frac{1}{2}\Delta + V_i + \sum_{j=1}^{p} k_{ij}|\psi_j^n|^2 - \omega_i L_z$ structurally mirrors the single-component case, the energy dissipation, global convergence, and local convergence properties can be straightforwardly deduced from the generalized abstract formulation presented in this work. From a methodological standpoint, earlier works interpreted GFSI as a nonlinear inverse iteration for eigenvalue problems \cite{ref_Henning}, thereby analyzing its convergence via the spectrum of the derivative operator of the iteration map at \textcolor{blue}{ground states}. In contrast, the present paper connects GFSI with \textcolor{blue}{Riemannian optimization, which is theoretically state-of-the-art for constrained optimization problems}. These comprehensive theoretical results fully resolve both the global and local convergence theories for this widely used numerical discretization of imaginary-time evolution.
\end{remark}

\section{Numerical experiments}

In this section, we numerically justify the energy-diminishing property given in Theorem \ref{THM} and the convergence results of the GFSI algorithm (\ref{original GFSI}). Always denote $k_m = \max \left\{|k_{11}|, |k_{12}|, |k_{22}| \right\}$.

As for the numerical settings, we consider the truncation ground state (\ref{mini problem}) with two-dimensional domain $\Omega = [-L,L] \times [-L,L]$ and harmonic external potential plus the strength of internal atomic Josephson junction:
\[
V_1(\mathbf{x}) = V_2(\mathbf{x}) = \frac{|\mathbf{x}|^2}{2} + |\beta|, \quad \mathbf{x} \in \Omega.
\]
 Furthermore, we numerically discretize the domain $\Omega$ to equidistant grid points in two directions, i.e. $h=h_x=h_y$. With regard to the semi-discrete scheme (\ref{original GFSI}) we adopt central differences to approximate the first and second derivatives. The iteration of this full-discretized GFSI algorithm is terminated when the following condition is fulfilled: $\frac{\|\Psi^{n+1}-\Psi^n\|_{\infty}}{\tau} < 10^{-7}$, where $\|\cdot\|_{\infty}$ denotes the discretized infinity norm and the resulting $\Psi^{n+1}$ is viewed as numerical ground state $\Psi^*$.
\begin{example}\label{example: energy decrease}
    In this numerical example, we examine the energy dissipation for interaction matrix $K$ being both positive-definite and all entries non-negative. Here we let  $L=4$, $\omega_1=0.5$, $\omega_2=0.5$, and initial data for GFSI algorithm are chosen as
    $(\Psi_0)_1 = (\Psi_0)_2 =  e^{-(x^2+y^2)/2}/\sqrt{2\pi} $. For these two cases we choose respectively.
\begin{enumerate}
\item[(i)] \textbf{Case 1}: all entries of interaction matrix $K$ are non-negative, specifically,  $k_{11}=100$, $k_{22}=97$, $k_{12}= 94$ and coupling strength $\beta=-5$.
\item[(ii)] \textbf{Case 2}: interaction matrix $K$ is positive-definite, specifically, $k_{11}=8.1$, $k_{22}=7.9$, $k_{12}= -0.94$ and coupling strength $\beta=0.2$.
\end{enumerate}
\end{example}
To distinguish the positive-definite case from all entries non-negative case, we specifically select $k_{12}$ as negative in \textbf{Case 2}. In these two cases we consider various time steps $\tau= 0.1$, $0.2$, $0.5$, $1.0$.  

\begin{figure}[ht]
  \centering
    \begin{minipage}{0.48\linewidth}
      \centering
      \includegraphics[height=3.5cm]{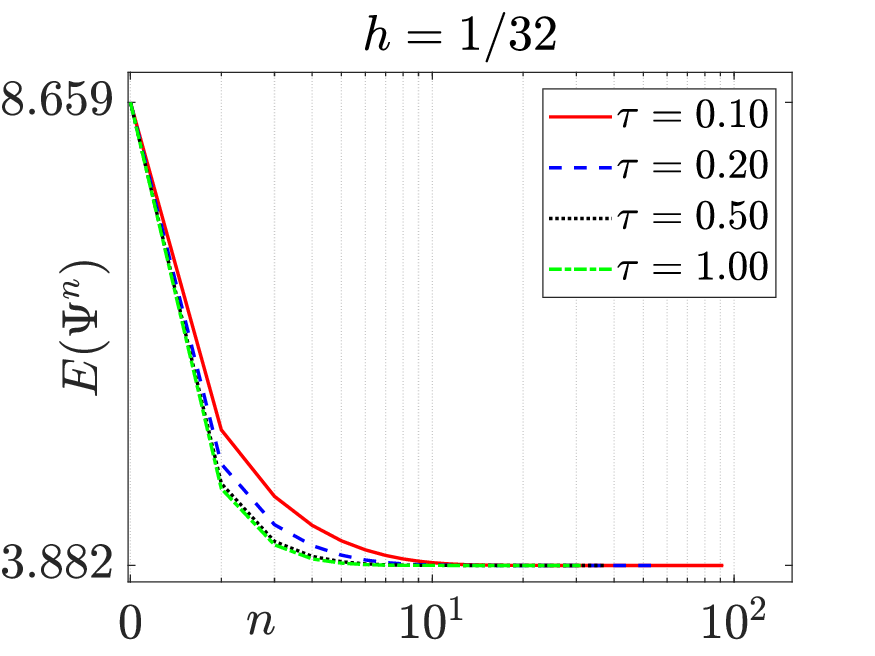}
    \end{minipage}\hfill
    \begin{minipage}{0.48\linewidth}
      \centering
      \includegraphics[height=3.5cm]{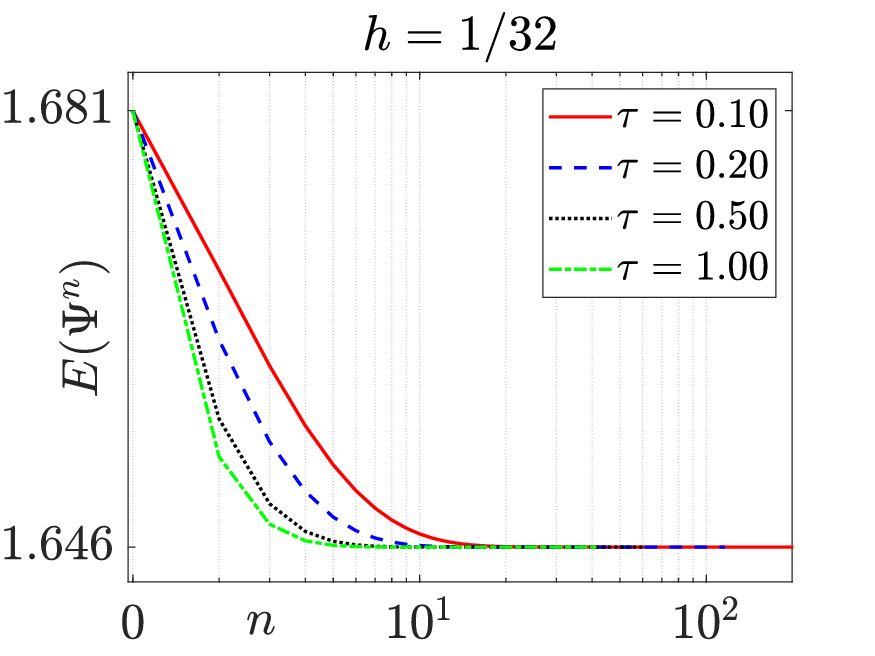}
    \end{minipage}
  \caption{Energy evolution that decrease at each step for various time steps $\tau= 0.1$, $0.2$, $0.5$, $1.0$ in \textbf{Case 1} (left) and \textbf{Case 2} (right) of Example~\ref{example: energy decrease}.}
  \label{fig:case12}
\end{figure}

Fig.\ref{fig:case12} shows the evolution of the energy for various time steps in both \textbf{Case 1} and \textbf{Case 2} the energy-diminishing property of the GFSI algorithm (\ref{original GFSI}) when time step $\tau \leq \tau_0$ is truly verified. The energy decay at each step numerically verifies the correctness of Theorem \ref{THM} we proved previously.

\begin{example}\label{exm2}
   In this numerical example, we demonstrate that for large $k_m$, the energy fails to exhibit the energy-diminishing property when the time step is not small. We choose the initial data as
 $(\Psi_0)_1 = (\Psi_0)_2 = \frac{e^{-(x^2+y^2)/2}}{\sqrt{2\pi}}$, set the rotation frequencies to $\omega_1 = \omega_2 = 0$, and fix the mesh size $h = 1/8$. We consider the following three cases.
\begin{enumerate}
    \item[(i)] \textbf{Case 3}: Interaction parameters $k_{11}: k_{22}: k_{12} = 2000: 2000: 1000$, coupling strength $\beta = -5$, $L = 4$ for various time steps $\tau = 1, 0.5, 0.2, 0.1$.
    \item[(ii)] \textbf{Case 4}: Interaction parameters $k_{11}: k_{22}: k_{12} = 5000: 5000: 4000$, coupling strength $\beta = 1$, $L = 4$ for various time steps $\tau = 1, 0.5, 0.2, 0.1$.
    \item[(iii)] \textbf{Case 5}: Interaction parameters $k_{11}: k_{22}: k_{12} = 10000: 10000: 8000$, coupling strength $\beta = -1$, $L = 8$ for various time steps $\tau = 1, 0.5, 0.2, 0.1$.
\end{enumerate}
Figure \ref{figures case3} displays the evolution of the total energy for these three cases. From the enlarged portions of the plots, it is evident that the energy-diminishing property is not preserved.
\end{example}
\vspace{-0.25cm}
\begin{figure}[ht]
  \centering
  \begin{minipage}{0.32\linewidth}
    \centering
    \includegraphics[width=\linewidth]{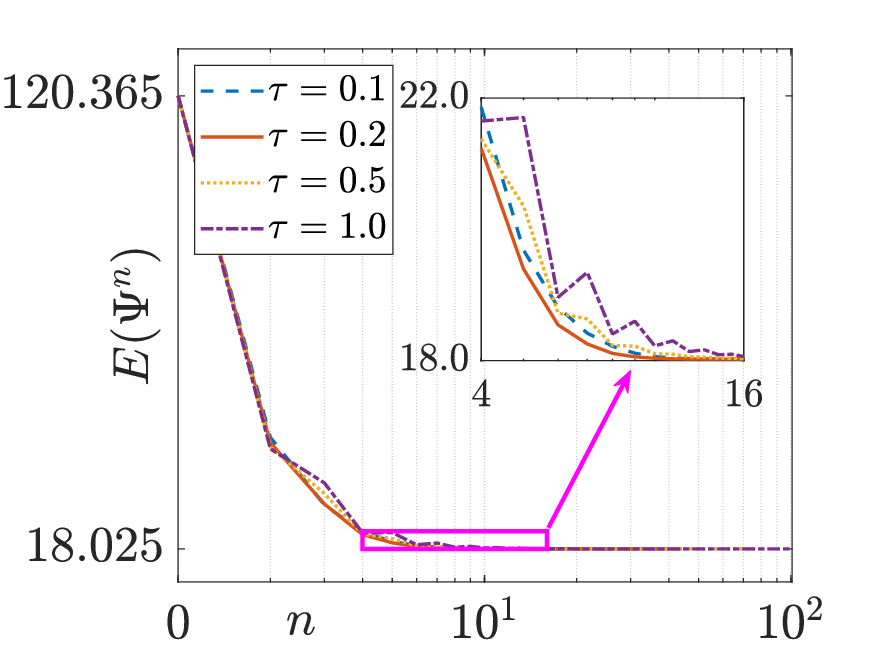}
  \end{minipage}\hfill
  \begin{minipage}{0.32\linewidth}
    \centering
    \includegraphics[width=\linewidth]{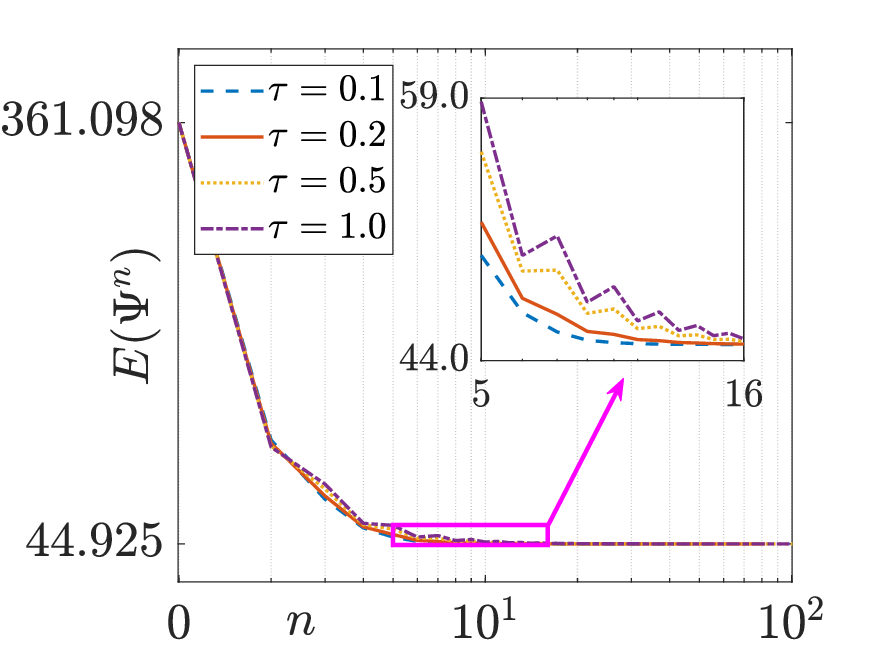}
  \end{minipage}\hfill
  \begin{minipage}{0.32\linewidth}
    \centering
    \includegraphics[width=\linewidth]{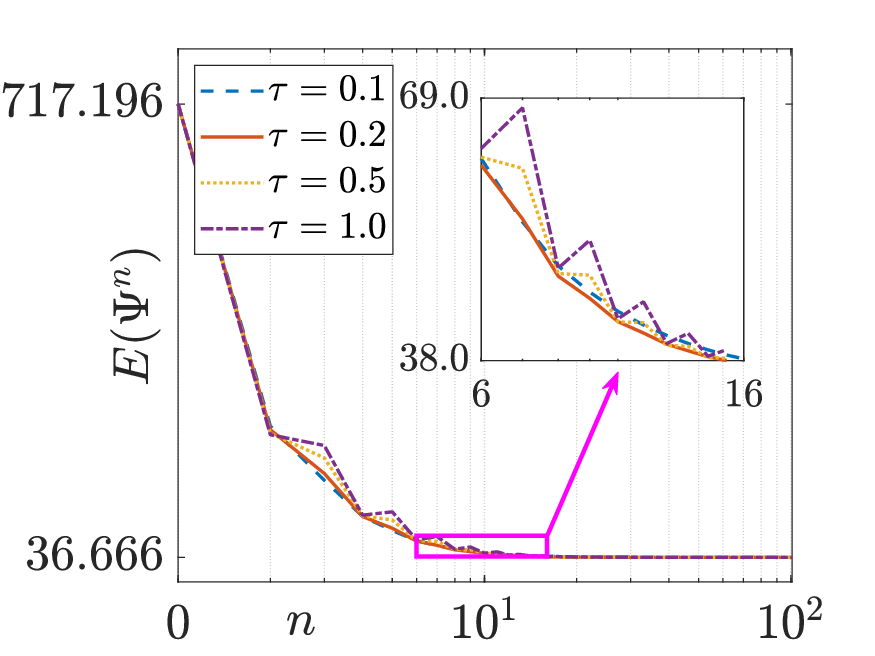}
  \end{minipage}
  \caption{The energy evolution does not satisfy energy-diminishing property in \textbf{Case 3}, \textbf{Case 4} and \textbf{Case 5} (from left to right) for mesh size $h = 1/8$ and various time steps $\tau$ of Example \ref{exm2}.}
  \label{figures case3}
\end{figure}
\vspace{-0.25cm}
\begin{example}\label{example: time step}
    In this numerical example, we test whether the upper bound of time steps $\tau_0$ in GFSI algorithm assuredly depends on the strength of particle interaction $k_m$. For this purpose, we let $\beta = -1$, $\omega_1 = \omega_2 = 0$, $L = 8$, mesh size $h = 1/8$ and initial data for GFSI $(\Psi_0)_1 = (\Psi_0)_2 = e^{-(x^2+y^2)/2}/\sqrt{2\pi}$. More importantly, choose the repulsive parameter $k_m$ ranging from $1000$ to $15000$ with $k_{11} = k_{22} = 5/4 k_{12}$, and time step $\tau = 1,0.8,0.6,0.4,0.2$.
\end{example}

Table \ref{table: example 2} lists the total energy of the converged ground state $\Psi^*$ for various time steps and numbers represent the iteration numbers while the underlined implies the sequence at this time step generated by full-discretized GFSI algorithm does not conform to the energy-diminishing property. We can tell from this numerical observation that appropriate upper bound of time steps decreases as $k_m$ increases, which is consistent with our result.
\vspace{-0.25cm}
\begin{table}[ht]
\caption{Total energy under different time steps in Example \ref{example: time step}. The underlined numbers indicate violations of the energy-diminishing property.}
\label{table: example 2}
\centering
\small 
\begin{tabular}{ccccccc} 
\hline
$E(\Psi^*)$ & $k_{11}$ & $\tau = 1.0$ & $\tau = 0.8$ & $\tau = 0.6$ & $\tau = 0.4$ & $\tau = 0.2$ \\
\hline
46.5819 & 15000 & \underline{265} & \underline{241} & \underline{209} & \underline{166} & 105 \\
36.6655 & 10000 & \underline{202} & \underline{184} & \underline{160} & 127 & 124 \\
25.3373 & 5000 & \underline{135} & \underline{123} & 108 & 110 & 122 \\ 
22.6311 & 4000 & \underline{120} & 109 & 96 & 102 & 114 \\
11.3646 & 1000 & 59 & 58 & 61 & 65 & 76 \\
\hline
\end{tabular}
\end{table}
\vspace{-0.25cm}
\begin{example}\label{exm4}
    In this numerical example, we plot the convergence profiles of the wave function differences $\|\Psi^{n+1}-\Psi^n\|_{\infty}$ and the residuals defined as 
    \[
       \mathrm{res}^n := \| \mathcal{H}_{\Psi^n}\Psi^n-\mu_{\Psi^n}\Psi^n \|_{\infty}, \quad \mu_{\Psi^n} = \langle \mathcal{H}_{\Psi^n}\Psi^n, \Psi^n \rangle. \]
    at each iteration step for following three cases. And we set the stopping criterion as $\|\Psi^{n+1}-\Psi^n\|_{\infty} /\tau < 10^{-10}$ in this example.
\begin{enumerate}
  \item[(i)] \textbf{Case 6}: Interaction parameters $k_{11}=500$, $k_{22}=97$, $k_{12}=53$, coupling strength $\beta = -5$, $\omega_1 = \omega_2 = 0.5$, mesh size $h = 1/16$ for $L = 8$ and time steps $\tau = 0.2$ with initial data chosen as $(\Psi_0)_1 = (\Psi_0)_2 =  (x+iy)e^{-(x^2+y^2)/2}/\sqrt{2\pi} $.
  \item[(ii)] \textbf{Case 7}: Interaction parameters $k_{11}=10$, $k_{22}=1.0$, $k_{12}=-0.97$, coupling strength $\beta = 5$, $\omega_1 = \omega_2 = 0.5$, mesh size $h = 1/16$ for $L = 5$ and time steps $\tau = 0.2$ with initial data chosen as $(\Psi_0)_1 = (\Psi_0)_2 =  e^{-(x^2+y^2)/2}/\sqrt{2\pi} $.
  \item[(iii)] \textbf{Case 8}: Interaction parameters $k_{11}=10$, $k_{22}=9.7$, $k_{12}=20$, coupling strength $\beta = 5$, $\omega_1 = 0.8$, $\omega_2 = 0.3$, mesh size $h = 1/16$ for $L = 8$ and time steps $\tau = 0.2$ with initial data chosen as $(\Psi_0)_1 = (\Psi_0)_2 =  (x+iy)e^{-(x^2+y^2)/2}/\sqrt{2\pi} $.
\end{enumerate}
\end{example}

As shown in Figure \ref{fig res}, the residual exhibits a linear convergence trend at later stages. Since the residual $\| \mathcal{H}_{\Psi^n}\Psi^n-\mu_{\Psi^n}\Psi^n \|_{\infty}$ and $\|\Psi^{n+1}-\Psi^n\|_{\infty}$ are of the same order, we conclude that the algorithm has entered the local convergence regime and converges to some stationary state described in Theorem \ref{global convergence}. \textcolor{blue}{To numerically examine the Morse-Bott condition in \eqref{Ker-Hess}, Table \ref{table eig} lists $\lambda_g$ and the first five smallest eigenvalues $\lambda_l$ $(l = 1,\dots,5)$ of $E''(\Psi_g) |_{T_{\Psi_g}\mathcal{M}}$. The continuous symmetry of the system comprises both phase shift and rotation; however, as illustrated by the contour plots of the numerical ground states $\Psi^* = (\psi_1, \psi_2)^{\top}$ for \textbf{Case 6} and \textbf{Case 7} in Fig.~\ref{contour}, these two transformations coincide on the obtained solutions due to their rotational invariance. Consequently, only a single zero mode exists, which is confirmed by $\lambda_g = \lambda_1$.} 

\textcolor{blue}{Then we precisely verify the convergence rate described in Corollary \ref{coroll2}, we compute the spectral radius $\rho_0$ given by \eqref{rho-less-1} and compare it with the actual residual convergence in Figure \ref{fig:case68}.}

\begin{figure}[htbp]
  \centering
  \begin{minipage}{0.32\linewidth}
    \centering
    \includegraphics[width=\linewidth]{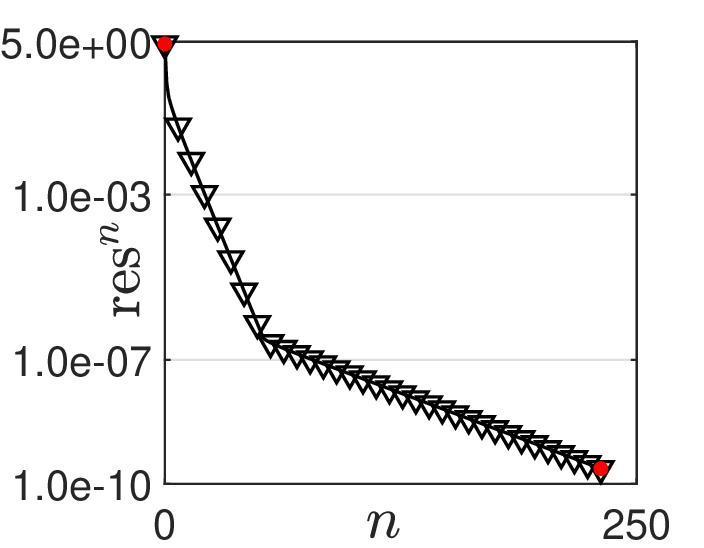}
  \end{minipage}\hfill
  \begin{minipage}{0.32\linewidth}
    \centering
    \includegraphics[width=\linewidth]{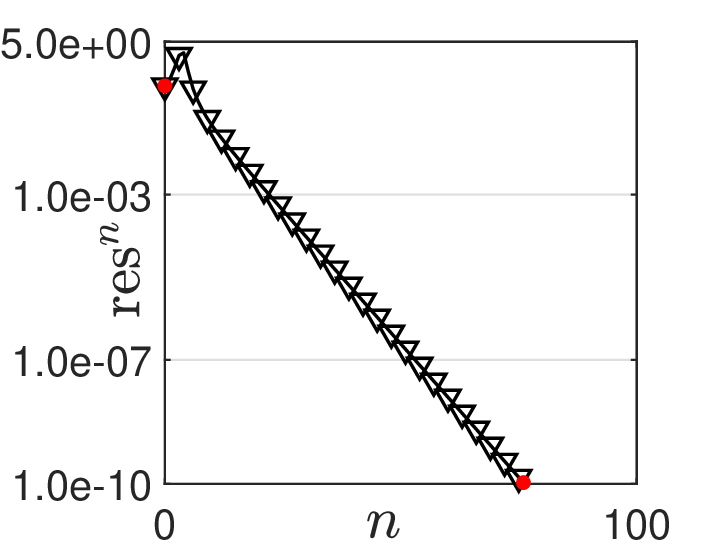}
  \end{minipage}\hfill
  \begin{minipage}{0.32\linewidth}
    \centering
    \includegraphics[width=\linewidth]{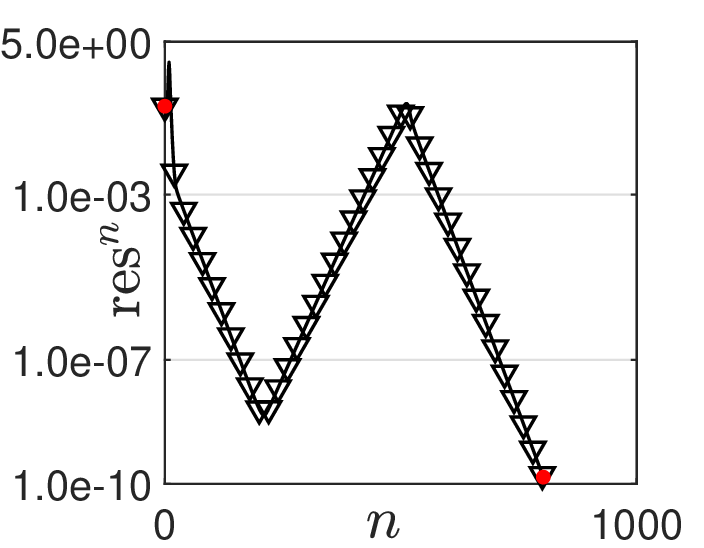}
  \end{minipage}
  \caption{The residuals $\mathrm{res}^n$ of each iteration of \textbf{Case 6}, \textbf{Case 7} and \textbf{Case 8} (from left to right).}
  \label{fig res}
\end{figure}

\begin{table}[htbp]
\caption{\textcolor{blue}{The values of $\lambda_g$ and the first five smallest eigenvalues of $E''(\Psi_g) \mid _{T_{\Psi_g}\mathcal{M}}$ in Example \ref{exm4} \textbf{case 6}.}}
\label{table eig}
\centering
\small
\vspace{-0.25cm}
\begin{tabular}{cccccc} 
\toprule
$\lambda_g$ & $\lambda_1$ & $\lambda_2$ & $\lambda_3$ & $\lambda_4$ & $\lambda_5$ \\ 
\midrule
6.350345 & 6.350345 & 6.354842 & 6.384984 & 6.708305 & 6.708665 \\ 
\bottomrule
\end{tabular}
\end{table}
\vspace{-0.25cm}
\begin{figure}[htbp]
  \centering
  \includegraphics[scale=0.5]{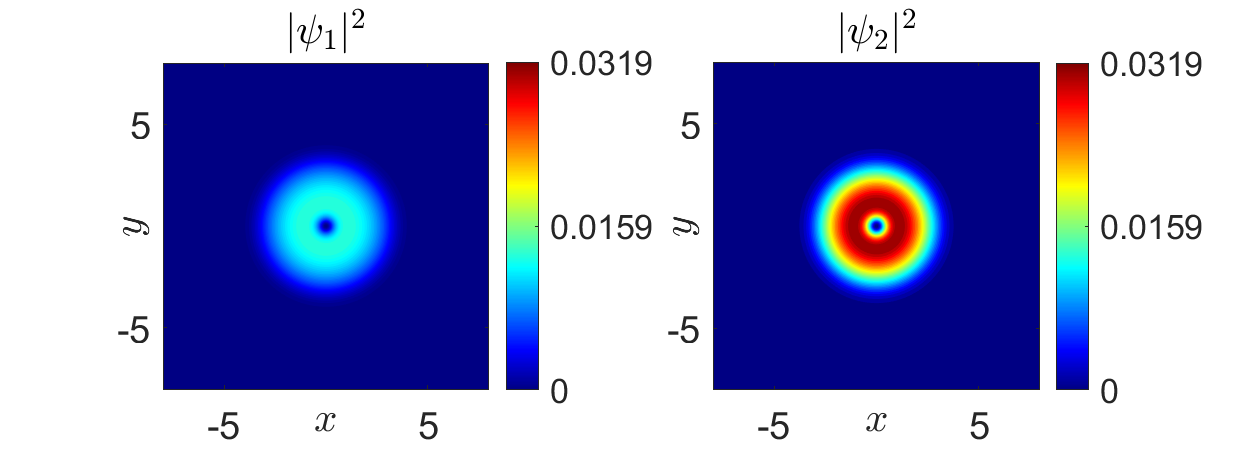}
  \vspace{-1em}  
  \includegraphics[scale=0.5]{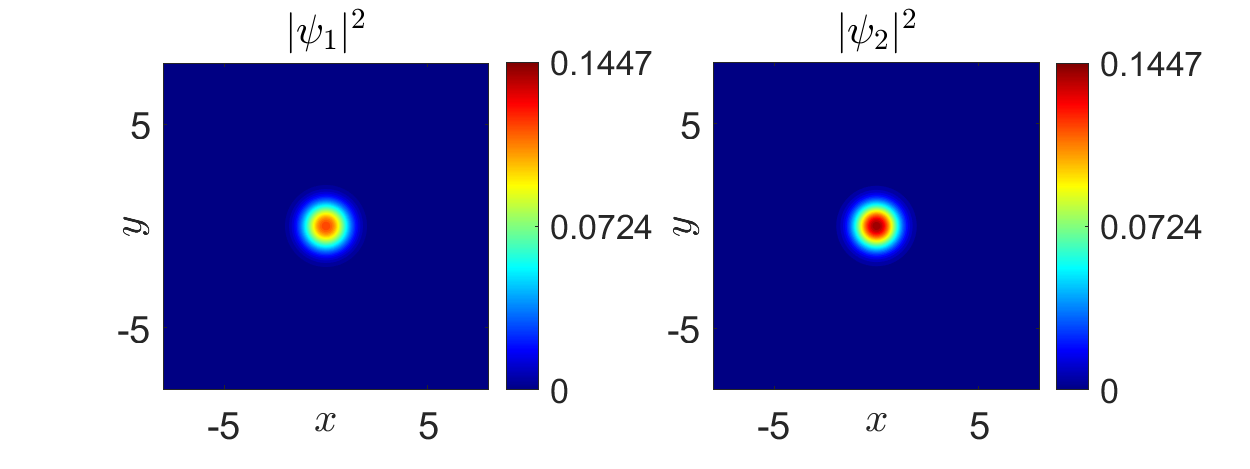} 
  \caption{Contour plots of the converged functions $|\psi_1|^2$ and $|\psi_2|^2$  in \textbf{Case 6} (upper) and \textbf{Case 7} (lower) of Example~\ref{example: energy decrease}.}
  \label{contour}
\end{figure}

\begin{figure}[htbp]
   \centering
  \begin{minipage}{0.32\linewidth}
    \centering
    \includegraphics[width=\linewidth]{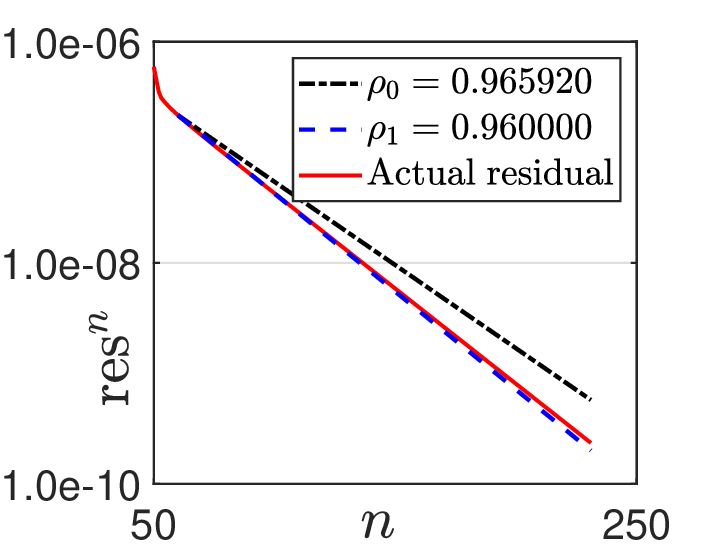}
  \end{minipage}\hfill
  \begin{minipage}{0.32\linewidth}
    \centering
    \includegraphics[width=\linewidth]{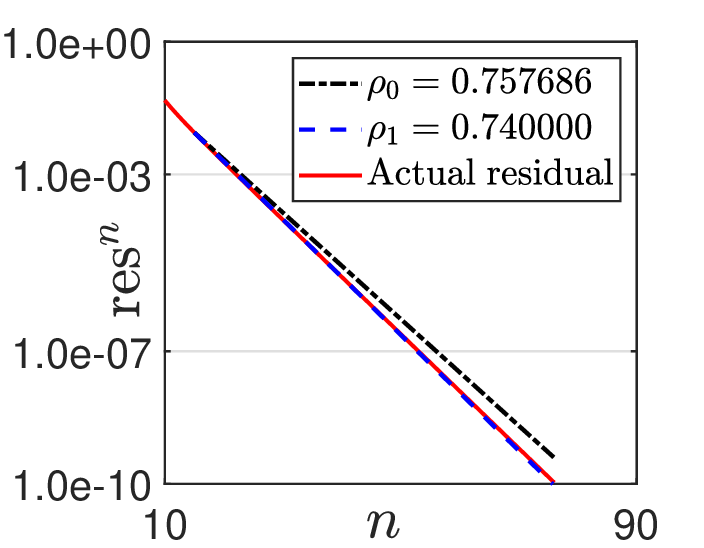}
  \end{minipage}\hfill
  \begin{minipage}{0.32\linewidth}
    \centering
    \includegraphics[width=\linewidth]{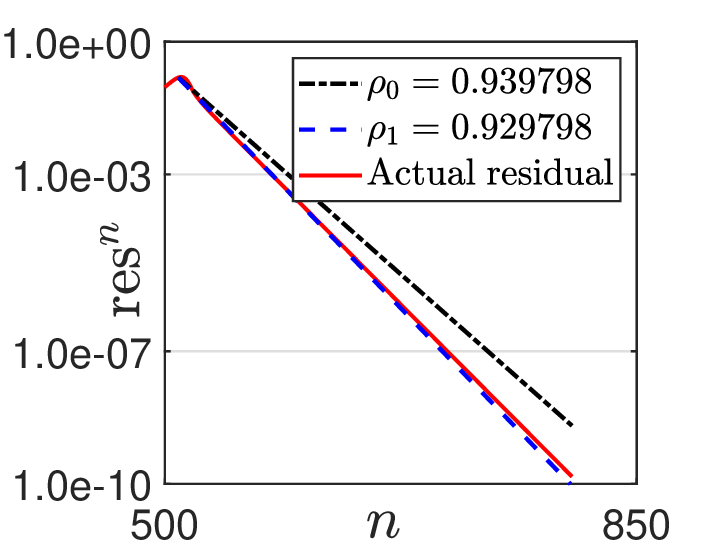}
  \end{minipage}
  \caption{\textcolor{blue}{Comparison between spectral radius $\rho_0$ and linear convergence rate for \textbf{case 6}, \textbf{Case 7} and \textbf{case 8} (from left to right) in Example \ref{exm4}.}}
  \label{fig:case68}
\end{figure}

\section{Conclusion}

In this paper, we have established a comprehensive and rigorous global-local convergence theory for the gradient flow with semi-implicit discretization (GFSI) applied to two-component Bose-Einstein condensates incorporating both Josephson junction and rotating term. This work unifies classical imaginary-time evolution discretizations in mathematical physics with modern optimization on manifolds within a common theoretical framework. First, by introducing a novel Lagrange multiplier reformulation, we proved that the discrete scheme strictly preserves the energy dissipation property, which in turn guarantees global convergence of the iteration sequence to constrained stationary states. Second, our theoretical results establish a fundamental equivalence between GFSI and a specific form of preconditioned Riemannian gradient descent. By interpreting the discrete normalized gradient flow within the framework of Riemannian optimization, we provide a sharp local convergence behavior of GFSI. Also, extensive numerical experiments confirmed our theoretical predictions.


\bibliographystyle{siamplain}
\bibliography{references}

\appendix

\section{Proof of Lemma \ref{existence of ground state}}
\label{proof of EXT}

\begin{proof}

According to (\ref{original total energy}) and the coercivity of $\mathcal{H}_{\mathbf{0}}$ (\ref{H0}) we obtain that
\[
E(\Psi)=\langle \mathcal{H}_{\mathbf{0}} \Psi,\Psi \rangle + \sum_{1,2}\int_{\Omega} \frac{1}{2}|\rho_i(\Psi)||\psi_i|^2\mathrm{d}\mathbf{x} \geq C_0\|\Psi\|_{H^1_0}^2 \geq 0,
\]
which means energy functional $E$ is bounded below, then there exists a minimizing sequence $\{\Psi^m\}_{m \in \mathbb{N}} \subset \mathcal{M}$. 

As such sequence is bounded in $[H_0^1(\Omega)]^2$, there exists a $\hat{\Psi}$ in $H$ such that
\[
  \psi_i^{m} \rightharpoonup \hat{\Psi}_i \quad i=1,2 \quad \text{weakly in} \quad H_0^1(\Omega).  
\]
The compact embedding $H_0^1(\Omega) \hookrightarrow L^2(\Omega)$ implies
$$
\psi_i^{m} \to \hat{\Psi}_i \quad i=1,2 \quad \text{strongly in} \quad L^2(\Omega).
$$
Then $ \|\hat{\Psi}\|_{L^2}=\lim_{m \rightarrow \infty} \|\Psi^m\|_{L^2} = 1$ further indicates $\hat{\Psi} \in \mathcal{M}$.

Discussion above yields $\mathcal{M}$ is a bounded weakly closed subset, associated with the weak lower semi-continuity of $E$ in $[H_0^1(\Omega)]^2$ immediately shows the existence of $ \arg\min_{\Psi \in \mathcal{M}} E(\Psi)$.
\end{proof}

\end{document}